\documentclass[11pt]{amsart}

\usepackage[utf8]{inputenc}
\usepackage[T1]{fontenc}

\usepackage{amsmath,amssymb,amsfonts,textcomp,amsthm,xifthen,psfrag,graphicx,color,pgfplots}
\usepackage{ifthen}

\usepackage{enumerate}

\usepackage{fullpage}

\usepackage[utf8]{inputenc}

\usepackage{hyperref}

\usepackage[capitalize,nameinlink]{cleveref}[0.19]
\crefname{section}{section}{sections}
\crefname{subsection}{subsection}{subsections}
\Crefname{section}{Section}{Sections}
\Crefname{subsection}{Subsection}{Subsections}
\Crefname{figure}{Figure}{Figures}
\crefformat{equation}{\textup{#2(#1)#3}}
\crefrangeformat{equation}{\textup{#3(#1)#4--#5(#2)#6}}
\crefmultiformat{equation}{\textup{#2(#1)#3}}{ and \textup{#2(#1)#3}}
{, \textup{#2(#1)#3}}{, and \textup{#2(#1)#3}}
\crefrangemultiformat{equation}{\textup{#3(#1)#4--#5(#2)#6}}%
{ and \textup{#3(#1)#4--#5(#2)#6}}{, \textup{#3(#1)#4--#5(#2)#6}}{, and \textup{#3(#1)#4--#5(#2)#6}}
\Crefformat{equation}{#2Equation~\textup{(#1)}#3}
\Crefrangeformat{equation}{Equations~\textup{#3(#1)#4--#5(#2)#6}}
\Crefmultiformat{equation}{Equations~\textup{#2(#1)#3}}{ and \textup{#2(#1)#3}}
{, \textup{#2(#1)#3}}{, and \textup{#2(#1)#3}}
\Crefrangemultiformat{equation}{Equations~\textup{#3(#1)#4--#5(#2)#6}}%
{ and \textup{#3(#1)#4--#5(#2)#6}}{, \textup{#3(#1)#4--#5(#2)#6}}{, and \textup{#3(#1)#4--#5(#2)#6}}
\crefdefaultlabelformat{#2\textup{#1}#3}

\usepackage{amsfonts,amssymb}
\usepackage{graphicx}
\usepackage{epstopdf}
\usepackage{algorithmic}
\ifpdf
  \DeclareGraphicsExtensions{.eps,.pdf,.png,.jpg}
\else
  \DeclareGraphicsExtensions{.eps}
\fi

\usepackage{pgfplots}
\usepgfplotslibrary{external}
\usepgfplotslibrary{colorbrewer}
\newtheorem{theorem}{Theorem}
\newtheorem{lemma}[theorem]{Lemma}

\newtheorem{remark}[theorem]{Remark}

\usepackage{amsopn}

\DeclareMathOperator{\diam}{diam}
\DeclareMathOperator{\dist}{dist}

\newcommand{\R}{\mathbb{R}}
\newcommand{\N}{\mathbb{N}}

\newcommand{\CF}{C_F}
\newcommand{\ssigma}{\boldsymbol{\sigma}}
\newcommand{\LL}{\boldsymbol{L}}
\DeclareMathOperator{\divergence}{div}

\newcommand{\uu}{\boldsymbol{u}}
\newcommand{\ww}{\boldsymbol{w}}
\newcommand{\vv}{\boldsymbol{v}}
\newcommand{\ttau}{{\boldsymbol\tau}}
\newcommand{\cchi}{{\boldsymbol\chi}}
\newcommand{\norm}[3][]{#1\|#2#1\|_{#3}}

\newcommand{\pwnabla}{\nabla_\TT}
\newcommand{\ip}[2]{(#1\hspace*{.5mm},#2)}
\newcommand{\dual}[2]{\langle#1\hspace*{.5mm},#2\rangle}

\newcommand{\TT}{\mathcal{T}}
\newcommand{\OO}{\mathcal{O}}
\newcommand{\NN}{\mathcal{N}}
\newcommand{\cS}{\mathcal{S}}
\newcommand{\PP}{\mathcal{P}}
\newcommand{\RT}{\mathcal{R}\!\mathcal{T}}

\newcommand{\Pidiv}{\Pi^{\divergence}}

\newcommand{\set}[2]{\big\{#1\,:\,#2\big\}}

\newcommand{\Copt}{C_\mathrm{opt}}
\newcommand{\Capp}{C_\mathrm{app}}
\newcommand{\Crel}{C_\mathrm{rel}}

\newcommand{\blfa}{a_\beta}
\newcommand{\blfb}{b_\beta}
\newcommand{\blfc}{c_\beta}

\DeclareMathOperator{\osc}{osc}
\newcommand{\oscf}{\osc}

\DeclareMathOperator{\est}{est}
\DeclareMathOperator{\err}{err}

\newcommand{\logLogSlopeTriangle}[6]
{

    \pgfplotsextra
    {
        \pgfkeysgetvalue{/pgfplots/xmin}{\xmin}
        \pgfkeysgetvalue{/pgfplots/xmax}{\xmax}
        \pgfkeysgetvalue{/pgfplots/ymin}{\ymin}
        \pgfkeysgetvalue{/pgfplots/ymax}{\ymax}

        \pgfmathsetmacro{\xArel}{#1}
        \pgfmathsetmacro{\yArel}{#3}
        \pgfmathsetmacro{\xBrel}{#1-#2}
        \pgfmathsetmacro{\yBrel}{\yArel}
        \pgfmathsetmacro{\xCrel}{\xArel}

        \pgfmathsetmacro{\lnxB}{\xmin*(1-(#1-#2))+\xmax*(#1-#2)} 
        \pgfmathsetmacro{\lnxA}{\xmin*(1-#1)+\xmax*#1} 
        \pgfmathsetmacro{\lnyA}{\ymin*(1-#3)+\ymax*#3} 
        \pgfmathsetmacro{\lnyC}{\lnyA+#4*(\lnxA-\lnxB)}
        \pgfmathsetmacro{\yCrel}{\lnyC-\ymin)/(\ymax-\ymin)} 

        \coordinate (A) at (rel axis cs:\xArel,\yArel);
        \coordinate (B) at (rel axis cs:\xBrel,\yCrel);
        \coordinate (C) at (rel axis cs:\xCrel,\yCrel);

        \draw[#5]   (A)--
                    (B)-- 
                    (C)-- node[pos=0.5,anchor=west] {#6}
                    cycle;
    }
}

\newcommand{\logLogSlopeTriangleBelow}[6]
{

    \pgfplotsextra
    {
        \pgfkeysgetvalue{/pgfplots/xmin}{\xmin}
        \pgfkeysgetvalue{/pgfplots/xmax}{\xmax}
        \pgfkeysgetvalue{/pgfplots/ymin}{\ymin}
        \pgfkeysgetvalue{/pgfplots/ymax}{\ymax}

        \pgfmathsetmacro{\xArel}{#1}
        \pgfmathsetmacro{\yArel}{#3}
        \pgfmathsetmacro{\xBrel}{#1-#2}
        \pgfmathsetmacro{\yBrel}{\yArel}
        \pgfmathsetmacro{\xCrel}{\xArel}

        \pgfmathsetmacro{\lnxB}{\xmin*(1-(#1-#2))+\xmax*(#1-#2)} 
        \pgfmathsetmacro{\lnxA}{\xmin*(1-#1)+\xmax*#1} 
        \pgfmathsetmacro{\lnyA}{\ymin*(1-#3)+\ymax*#3} 
        \pgfmathsetmacro{\lnyC}{\lnyA+#4*(\lnxA-\lnxB)}
        \pgfmathsetmacro{\yCrel}{\lnyC-\ymin)/(\ymax-\ymin)} 

        \coordinate (A) at (rel axis cs:\xArel,\yArel);
        \coordinate (B) at (rel axis cs:\xBrel,\yCrel);
        \coordinate (C) at (rel axis cs:\xBrel,\yArel);

        \draw[#5]   (A)--
                    (B)-- node[pos=0.5,anchor=east] {#6}
                    (C)-- 
                    cycle;
    }
}

\begin{document}

  \title[FOLS for the obstacle problem]{First-Order Least-Squares Method for the Obstacle problem}
\date{\today}

\author{Thomas F\"{u}hrer}
\address{Facultad de Matem\'{a}ticas, Pontificia Universidad Cat\'{o}lica de Chile, Santiago, Chile}
\email{tofuhrer@mat.uc.cl}

\thanks{{\bf Acknowledgment.} 
This work was supported by CONICYT through FONDECYT project ``Least-squares methods for obstacle problems'' under grant 11170050.}


\begin{abstract}
  We define and analyse a least-squares finite element method for a first-order reformulation of the obstacle problem.
  Moreover, we derive variational inequalities that are based on similar but non-symmetric bilinear forms.
  A priori error estimates including the case of non-conforming convex sets are given 
  and optimal convergence rates are shown for the lowest-order case.
  We provide also a posteriori bounds that can be be used as error indicators in an adaptive algorithm.
  Numerical studies are presented.
\end{abstract}

\maketitle

\section{Introduction}
Many physical problems are of obstacle type, or more generally, described by variational
inequalities~\cite{KinderlehrerStampacchia,Rodrigues_87_OPM}.
In this article we consider, as a model problem,
the classical obstacle problem where one seeks the equilibrium position of an elastic
membrane constrained to lie over an obstacle.

This type of problems is challenging, in particular for numerical methods, 
since solutions usually suffer from regularity issues and since the contact boundary is a priori unknown.
There exists already a long history of numerical methods, in particular finite element methods, see,
e.g., the books~\cite{Glowinski08,Glowinski81} for an overview on the topic.
However, the literature on least-squares methods for obstacle problems is scarce.
In fact, until the writing of this paper only~\cite{BurmanHLSLSQobstacle} was available for the classical obstacle problem 
where the idea goes back to a Nitsche-based method for contact problems introduced and analyzed in~\cite{ChoulyHildNitsche}.
An analysis of first-order least-squares finite element methods for Signorini problems can be found in~\cite{StarkeLSQSignoriniS}
and more recently~\cite{StarkeLSQSignorini}.
Let us also mention the pioneering work~\cite{Falk74} for the a priori analysis of a classical finite element scheme.
Newer articles include~\cite{GustafssonStenbergVideman_MixedFEMObstacle,MR3667082} where mixed and stabilized
methods are considered.

Least-squares finite element methods are a widespread class of numerical schemes and their basic idea is to approximate
the solution by minimizing a functional, e.g., the residual in some given norm.
Let us recall some important properties of least-squares finite element methods, a more 
complete list is given in the introduction of the overview article~\cite{BochevGunzbergerOverview}, see also the
book~\cite{BochevGunzburgerLSQ}.
\begin{itemize}
  \item \emph{Unconstrained stability:} One feature of least-squares schemes is that the methods are stable for all
    pairings of discrete spaces. 
  \item \emph{Adaptivity:} Another feature is that a posteriori bounds on the error are obtained by simply evaluating
    the least-squares functional.
    For instance, standard least-squares methods for the Poisson problem~\cite{BochevGunzburgerLSQ} are based on minimizing 
    residuals in $L^2$ norms, which can be localized and, then, be used as error indicators in an adaptive algorithm.
\end{itemize}
The main purpose of this paper is to close the gap in the literature and define least-squares based
methods for the obstacle problems. 
In particular, we want to study if the aforementioned properties transfer to the case of obstacle problems.
Let us shortly describe the functional our method is based on. For simplicity assume a zero obstacle (the remainder of
the paper deals with general non-zero obstacles). Then, the problem reads
\begin{align*}
  -\Delta u\geq f, \quad u\geq 0, \quad (-\Delta u-f)u = 0
\end{align*}
in some domain $\Omega$ and $u|_{\partial \Omega}=0$.
Introducing the Lagrange multiplier (or reaction force) $\lambda = -\Delta u-f$ and $\ssigma=\nabla u$, we rewrite the
problem as a first-order system, see also~\cite{BanzSchroeder,BanzStephan,BurmanHLSLSQobstacle,GustafssonStenbergVideman_MixedFEMObstacle},
\begin{align*}
  -\divergence\ssigma - \lambda = f, \quad \ssigma-\nabla u = 0, \quad u\geq 0, \quad \lambda\geq 0, 
  \quad \lambda u = 0.
\end{align*}
Note that $f\in L^2(\Omega)$ does not imply more regularity for $u$ so that $\lambda\in H^{-1}(\Omega)$ is only in the
dual space in general. However, observe that $\divergence\ssigma+\lambda=-f\in L^2(\Omega)$ and therefore the functional
\begin{align*}
  J( (u,\ssigma,\lambda); f) := \norm{\divergence\ssigma+\lambda+f}{}^2 + \norm{\nabla u-\ssigma}{}^2 +
  \dual{\lambda}u,
\end{align*}
where $\dual\cdot\cdot$ denotes a duality pairing, is well-defined for $\divergence\ssigma+\lambda\in L^2(\Omega)$.
We will show that minimizing $J$ over a convex set with the additional
linear constraints $u\geq0$, $\lambda\geq 0$ is equivalent to solving the obstacle problem.
We will consider the variational inequality associated to this problem with corresponding bilinear form
$a(\cdot,\cdot)$. An issue that arises is that $a(\cdot,\cdot)$ is not necessarily coercive.
However, as it turns out, a simple scaling of the first term in the functional ensures coercivity on the whole space.
In view of the aforementioned properties, this means that our method is \emph{unconstrained stable}.
The recent work~\cite{GustafssonStenbergVideman_MixedFEMObstacle} based on a Lagrange formulation (without reformulation
to a first-order system) considers augmenting the trial spaces with bubble functions (mixed method) resp. adding residual
terms (stabilized method) to obtain stability.

Furthermore, we will see that the functional $J$ evaluated at some discrete approximation 
$(u_h,\ssigma_h,\lambda_h)$ with $u_h,\lambda_h\geq 0$
is an upper bound for the error. Note that for $\lambda_h\in L^2(\Omega)$ the duality
$\dual{\lambda_h}{u_h}$ reduces to the $L^2$ inner product. Thus, all the terms in the functional can be localized and
used as error indicators.

Additionally, we will derive and analyse other variational inequalities that are also based on the first-order
reformulation.
The resulting methods are quite similar to the least-squares scheme since they share the same residual terms.
The only difference is that the compatibility condition $\lambda u = 0$ is incorporated in a different, non-symmetric,
way.
We will present a uniform analysis that covers the least-squares formulation and the novel variational inequalities
of the obstacle problem.

Finally, we point out that the use of adaptive schemes for obstacle problems is quite natural.
First, the solutions may suffer from singularities stemming from the geometry, and
second, the free boundary is a priori unknown.
There exists plenty of literature on a posteriori estimators resp. adaptivity for finite elements methods for the obstacle
problem, see, 
e.g.~\cite{Braess05,CarstensenBartels04,ChenNochetto00,NochettoSiebertVeeser05,NochettoSV03,Veeser01I,WeissWohlmuth10} 
to name a few.
Many of the estimators are based on the use of a discrete Lagrange multiplier which is obtained in a postprocessing step.
In contrast, our proposed methods simultaneously approximate the Lagrange multiplier. 
This allows for a simple analysis of reliable a posteriori bounds.

\subsection{Outline}
The remainder of the paper is organized as follows. 
In~\cref{sec:main} we describe the model problem, introduce the corresponding first-order system and based on that
reformulation define our least-squares method.
Then,~\cref{sec:VI} deals with the definition and analysis of different variational inequalities.
In~\cref{sec:apost} we provide an a posteriori analysis and numerical studies are presented in~\cref{sec:examples}.
Some concluding remarks are given in~\cref{sec:conclusions}.

\section{Least-squares method}\label{sec:main}
In~\crefrange{sec:model}{sec:notation} we describe the model problem and introduce the reader to our notation.
Then,~\cref{sec:lsq} is devoted to the definition and analysis of a least-squares functional.

\subsection{Model problem}\label{sec:model}
Let $\Omega\subset \R^n$, $n=2,3$ denote a polygonal Lipschitz domain with boundary $\Gamma=\partial\Omega$.
For given $f\in L^2(\Omega)$ and $g\in H^1(\Omega)$ with $g|_{\Gamma}\leq 0$ we consider the classical obstacle problem:
Find a solution $u$ to
\begin{subequations}\label{eq:model}
\begin{alignat}{2}
  -\Delta u &\geq f &\quad&\text{in }\Omega, \\
  u&\geq g &\quad&\text{in }\Omega, \\
  (u-g)(-\Delta u-f) &= 0 &\quad&\text{in }\Omega, \\
  u &= 0 &\quad&\text{on }\Gamma.
\end{alignat}
\end{subequations}
It is well-known that this problem admits a unique solution $u\in H_0^1(\Omega)$, and it
can be equivalently characterized by the variational inequality: 
Find $u\in H_0^1(\Omega)$, $u\geq g$ such that
\begin{align}\label{eq:model:VI}
  \int_{\Omega} \nabla u \cdot\nabla(v-u)\, dx \geq \int_\Omega f(v-u) \,dx
  \quad\text{for all } v\in H_0^1(\Omega), v\geq g,
\end{align}
see~\cite{KinderlehrerStampacchia}. 
For a more detailed description of the involved function spaces we refer to~\cref{sec:notation} below.

\subsection{Notation \& function spaces}\label{sec:notation}
We use the common notation for Sobolev spaces $H_0^1(\Omega)$, $H^s(\Omega)$ ($s>0$).
Let $\ip\cdot\cdot$ denote the $L^2(\Omega)$ inner product, which induces the norm $\norm{\cdot}{}$.
The dual of $H_0^1(\Omega)$ is denoted by $H^{-1}(\Omega) := (H_0^1(\Omega))^*$, where duality 
$\dual\cdot\cdot$ is understood with respect to the extended $L^2(\Omega)$ inner product.
We equip $H^{-1}(\Omega)$ with the dual norm
\begin{align*}
  \norm{\lambda}{-1} := \sup_{0\neq v\in H_0^1(\Omega)} \frac{\dual{\lambda}v}{\norm{\nabla v}{}}.
\end{align*}

Recall Friedrichs' inequality
\begin{align*}
  \norm{u}{} \leq \CF \norm{\nabla v}{} \quad\text{for }v\in H_0^1(\Omega),
\end{align*}
where $0<\CF=\CF(\Omega)\leq \diam(\Omega)$. Thus, by definition we have $\norm{\lambda}{-1}\leq \CF \norm{\lambda}{}$ for
$\lambda\in L^2(\Omega)$.

Let $\divergence : \LL^2(\Omega):=L^2(\Omega)^n \to H^{-1}(\Omega)$ denote the generalized divergence operator, i.e.,
$\dual{\divergence\ssigma}{u} := -\ip{\ssigma}{\nabla u}$ for all $\ssigma\in \LL^2(\Omega)$, $u\in H_0^1(\Omega)$.
This operator is bounded,
\begin{align*}
  \norm{\divergence\ssigma}{-1} = \sup_{0\neq v\in H_0^1(\Omega)} 
  \frac{\dual{\divergence\ssigma}v}{\norm{\nabla v}{}}
  = \sup_{0\neq v\in H_0^1(\Omega)} \frac{-\ip{\ssigma}{\nabla v}}{\norm{\nabla v}{}} \leq
  \norm{\ssigma}{}.
\end{align*}

Let $v\in H^1(\Omega)$. We say $v\geq 0$ if $v\geq 0$ a.e.~in $\Omega$. Moreover, $\lambda\geq0$ for $\lambda\in
H^{-1}(\Omega)$ means that $\dual{\lambda}v\geq 0$ for all $v\in H_0^1(\Omega)$ with $v\geq 0$.

Define the space
\begin{align*}
  V := H_0^1(\Omega)\times \LL^2(\Omega) \times H^{-1}(\Omega)
\end{align*}
with norm
\begin{align*}
  \norm{\vv}V^2 := \norm{\nabla v}{}^2 + \norm{\ttau}{}^2 + \norm{\mu}{-1}^2 
  \quad\text{for } \vv = (v,\ttau,\mu) \in V
\end{align*}
and the space
\begin{align*}
  U := \set{(u,\ssigma,\lambda)\in V}{\divergence\ssigma+\lambda\in L^2(\Omega)}
\end{align*}
with norm
\begin{align*}
  \norm{\uu}U^2 := \norm{\nabla u}{}^2 + \norm{\ssigma}{}^2 + \norm{\divergence\ssigma+\lambda}{}^2
  \quad\text{for }\uu = (u,\ssigma,\lambda)\in U.
\end{align*}
Observe that $\norm{\cdot}U$ is a stronger norm than $\norm{\cdot}V$, i.e., 
\begin{align*}
  \norm{\nabla u}{}^2 + \norm{\ssigma}{}^2 + \norm{\lambda}{-1}^2 
  &\leq \norm{\nabla u}{}^2 + \norm{\ssigma}{}^2 + 2\norm{\divergence\ssigma+\lambda}{-1}^2
  +2\norm{\divergence\ssigma}{-1}^2  \\
  &\leq \norm{\nabla u}{}^2 + 3\norm{\ssigma}{}^2 +
  2\CF^2\norm{\divergence\ssigma+\lambda}{}^2.
\end{align*}

Our first least-squares formulation will be based on the minimization over the non-empty, convex and closed subset 
\begin{align*}
  K^{s} := \set{(u,\ssigma,\lambda)\in U}{ u-g\geq 0,\, \lambda\geq 0},
\end{align*}
where $g$ is the given obstacle function.
We will also derive and analyse variational inequalities based on non-symmetric bilinear forms that utilize the sets
\begin{align*}
  K^{0} &:= \set{(u,\ssigma,\lambda)\in U}{ u-g\geq 0}, \\
  K^{1} &:= \set{(u,\ssigma,\lambda)\in U}{ \lambda\geq 0}.
\end{align*}
Clearly, $K^s\subset K^{j}$ for $j=1,2$.

We write $A\lesssim B$ if there exists a constant $C>0$, independent of quantities of interest, such that
$A\leq C B$. 
Analogously we define $A\gtrsim B$.
If $A\lesssim B$ and $B\lesssim A$ holds then we write $A\simeq B$.

\subsection{Least-squares functional}\label{sec:lsq}
Let $u\in H_0^1(\Omega)$ denote the unique solution of the obstacle problem~\eqref{eq:model}.
Define $\lambda := -\Delta u - f\in H^{-1}(\Omega)$ and $\ssigma:=\nabla u$. 
Problem~\eqref{eq:model} can equivalently be written as the first-order problem
\begin{subequations}\label{eq:model:fo}
\begin{alignat}{2}
  -\divergence\ssigma-\lambda &= f &\quad&\text{in }\Omega, \\
  \ssigma-\nabla u &= 0 &\quad&\text{in }\Omega, \\
  u&\geq g &\quad&\text{in }\Omega, \\
  \lambda&\geq 0 &\quad&\text{in }\Omega, \\
  (u-g)\lambda &= 0 &\quad&\text{in }\Omega, \\
  u &= 0 &\quad&\text{on }\Gamma.
\end{alignat}
\end{subequations}
Observe that $\divergence\ssigma+\lambda\in L^2(\Omega)$ and that the unique solution $\uu =(u,\ssigma,\lambda)\in U$
satisfies $\uu\in K^s$.
We consider the functional
\begin{align*}
  J(\uu;f,g) := \norm{\divergence\ssigma+\lambda+f}{}^2 + \norm{\nabla u-\ssigma}{}^2 + \dual{\lambda}{u-g}
\end{align*}
for $\uu=(u,\ssigma,\lambda)\in U$, $f\in L^2(\Omega)$, $g\in H_0^1(\Omega)$ 
and the minimization problem: Find $\uu\in K^s$ with
\begin{align}\label{eq:lsq}
  J(\uu;f,g) = \min_{\vv\in K^s} J(\vv;f,g).
\end{align}
Note that the definition of the functional only makes sense if $g\in H_0^1(\Omega)$.

\begin{theorem}\label{thm:lsq}
  If $f\in L^2(\Omega)$, $g\in H_0^1(\Omega)$, then problems~\eqref{eq:model:fo} and~\eqref{eq:lsq} are equivalent.
  In particular, there exists a unique solution $\uu\in K^s$ of~\eqref{eq:lsq} and it holds that
  \begin{align}\label{eq:lsq:error}
    J(\vv;f,g) \geq C_J \norm{\vv-\uu}U^2 \quad\text{for all }\vv\in K^s.
  \end{align}
  The constant $C_J>0$ depends only on $\Omega$.
\end{theorem}
\begin{proof}
  Let $\uu := (u,\ssigma,\lambda) =(u,\nabla u,-\Delta u - f)\in K^s$ denote the unique solution of~\eqref{eq:model:fo}.
  Observe that $J(\vv;f,g)\geq 0$ for all $\vv\in K^s$ and $J(\uu;f,g)=0$, thus, $\uu$ minimizes the functional.
  Suppose~\eqref{eq:lsq:error} holds and that $\uu^*\in K^s$ is another minimizer. Then,~\eqref{eq:lsq:error} proves
  that $\uu=\uu^*$. 
  It only remains to show~\eqref{eq:lsq:error}.
  Let $\vv=(v,\ttau,\mu)\in K^s$. Since $f=-\divergence\ssigma-\lambda$ and $\nabla u-\ssigma = 0$ we have with the
  constant $\CF>0$ that
  \begin{align*}
    J(\vv;f,g) &= \norm{\divergence(\ttau-\ssigma)+(\mu-\lambda)}{}^2 + \norm{\nabla(v-u)-(\ttau-\ssigma)}{}^2
    + \dual{\mu}{v-g} \\
    &\simeq (1+\CF^2) \norm{\divergence(\ttau-\ssigma)+(\mu-\lambda)}{}^2 + \norm{\nabla(v-u)-(\ttau-\ssigma)}{}^2
    + \dual{\mu}{v-g}.
  \end{align*}
  Moreover, $\dual{\lambda}{u-g}=0$ and $\dual{\lambda}{v-g}\geq 0$, $\dual{\mu}{u-g}\geq 0$. 
  Therefore,
  \begin{align*}
    \dual{\mu}{v-g} &= \dual{\mu}{v-u} + \dual{\mu}{u-g} + \dual{\lambda}{u-g} \\
    &\geq \dual{\mu}{v-u} + \dual{\lambda}{u-g} + \dual{\lambda}{g-v} 
    \\
    &= \dual{\mu}{v-u} + \dual{\lambda}{u-v} = \dual{\mu-\lambda}{v-u}.
  \end{align*}
  Define $\ww := (w,\cchi,\nu) := \vv-\uu$. Then, the Cauchy-Schwarz inequality,
  Young's inequality and the definition of the divergence operator
  yield
  \begin{align*}
    J(\vv;f,g) &\simeq (1+\CF^2) \norm{\divergence(\ttau-\ssigma)+(\mu-\lambda)}{}^2 + \norm{\nabla(v-u)-(\ttau-\ssigma)}{}^2
    + \dual{\mu}{v-g} \\
    &\geq (1+\CF^2) \norm{\divergence\cchi+\nu}{}^2 + \norm{\nabla w-\cchi}{}^2
    + \dual{\nu}{w} \\
    &= (1+\CF^2) \norm{\divergence\cchi+\nu}{}^2 + \norm{\nabla w}{}^2 + \norm{\cchi}{}^2 
    -\ip{\nabla w}{\cchi} + \dual{\divergence\cchi}{w} + \dual{\nu}{w} \\
    &\geq (1+\CF^2) \norm{\divergence\cchi+\nu}{}^2 + \tfrac12\norm{\nabla w}{}^2 + \tfrac12\norm{\cchi}{}^2
    + \dual{\divergence\cchi+\nu}w.
  \end{align*}
  Application of the Cauchy-Schwarz inequality, Friedrichs'
  inequality and Young's inequality gives us for the last term and $\delta>0$
  \begin{align*}
    |\dual{\divergence\cchi+\nu}w|\leq \CF\norm{\divergence\cchi+\nu}{}\norm{\nabla w}{} \leq \CF^2 \frac{\delta^{-1}}2 
    \norm{\divergence\cchi+\nu}{}^2 + \frac{\delta}2 \norm{\nabla w}{}^2.
  \end{align*}
  Putting altogether and choosing $\delta=\tfrac12$ we end up with
  \begin{align*}
    J(\vv;f,g) &\simeq (1+\CF^2) \norm{\divergence\cchi+\nu}{}^2 + \norm{\nabla w-\cchi}{}^2
    + \dual{\mu}{v-g}
    \\
    &\geq (1+\CF^2) \norm{\divergence\cchi+\nu}{}^2 + \norm{\nabla w-\cchi}{}^2
    + \dual{\nu}{w} \\
    &\geq \norm{\divergence\cchi+\nu}{}^2 + \tfrac14\norm{\nabla w}{}^2 + \tfrac12\norm{\cchi}{}^2
    \simeq \norm{\ww}U^2 = \norm{\vv-\uu}U^2,
  \end{align*}
  which finishes the proof.
\end{proof}

\begin{remark}
  Note that~\eqref{eq:lsq:error} measures the error of any function $\vv\in K^s$, in particular, it can be used as a
  posteriori error estimator when $\vv\in K_h^s\subset K^s$ is a discrete approximation.
  However, in practice the condition $K_h^s \subset K^s$ is hard to realize in most cases. Below we introduce a simple
  scaling of the first term in the least-squares functional that allows us to prove coercivity of the associated
  bilinear form on the whole space $U$.
\end{remark}

For given $f\in L^2(\Omega)$, $g\in H_0^1(\Omega)$, and fixed parameter $\beta>0$
define the bilinear form $\blfa : U\times U \to \R$ and functional $F_\beta : U\to \R$ by
\begin{align}
  \blfa(\uu,\vv)&:= \beta\ip{\divergence\ssigma+\lambda}{\divergence\ttau+\mu} + \ip{\nabla u-\ssigma}{\nabla v-\ttau}
  + \tfrac12(\dual{\mu}u + \dual{\lambda}v), \\
  F_\beta(\vv) &:= -\beta\ip{f}{\divergence\ttau+\mu} + \tfrac12\dual{\mu}g
\end{align}
for all $\uu=(u,\ssigma,\lambda), \vv = (v,\ttau,\mu)\in U$.
We stress that $a_1(\cdot,\cdot)$ and $F_1(\cdot)$ induce the functional $J(\cdot;\cdot)$, i.e.,
\begin{align*}
  J(\uu;f,g) = a_1(\uu,\uu)-2F_1(\uu)+\ip{f}f.
\end{align*}
Since $J$ is differentiable it is well-known that the solution 
$\uu\in K^s$ of~\eqref{eq:lsq} satisfies the variational inequality
\begin{align}\label{eq:lsq:VI}
  a_1(\uu,\vv-\uu) &\geq F_1(\vv-\uu) \quad\text{for all } \vv\in K^s.
\end{align}
Conversely, if $J$ is also convex in $K^s$, then
any solution of~\eqref{eq:lsq:VI} solves~\eqref{eq:lsq}.
However, $J$ is convex on $K^s$ iff $a_1(\vv-\ww,\vv-\ww)\geq0$ for all $\vv,\ww\in K^s$, which is not true in general.
In~\cref{sec:VI} below we will show that for sufficiently large
$\beta >1$ the bilinear form $\blfa(\cdot,\cdot)$ is coercive, even on the whole space $U$.
This has the advantage that we can prove unique solvability of the continuous problem and its discretization
simultaneously.
More important, in practice this allows the use of non-conforming subsets $K_h^s\nsubseteq K_h$.

\section{Variational inequalities}\label{sec:VI}
In this section we introduce and analyse different variational inequalities.
The idea of including the compatibility condition in different ways has also been used in~\cite{DPGsignorini} to derive
DPG methods for contact problems.

We define the bilinear forms $\blfb,\blfc : U\times U\to \R$ and functionals $G_\beta$, $H_\beta$ by
\begin{align*}
  \blfb(\uu,\vv)&:= \beta\ip{\divergence\ssigma+\lambda}{\divergence\ttau+\mu} + \ip{\nabla u-\ssigma}{\nabla v-\ttau}
  + \dual{\lambda}v, \\
  \blfc(\uu,\vv)&:= \beta\ip{\divergence\ssigma+\lambda}{\divergence\ttau+\mu} + \ip{\nabla u-\ssigma}{\nabla v-\ttau}
  + \dual{\mu}u, \\
  G_\beta(\vv) &:= -\beta\ip{f}{\divergence\ttau+\mu} \\
  H_\beta(\vv) &:= -\beta\ip{f}{\divergence\ttau+\mu} + \dual{\mu}g.
\end{align*}

Let $\uu = (u,\ssigma,\lambda)\in K^s\subset K^j$ ($j=0,1$) denote the unique solution of~\eqref{eq:model:fo} with $f\in
L^2(\Omega)$, $g\in H_0^1(\Omega)$.
Recall that $\divergence\ssigma+\lambda = -f$. Testing this identity with $\divergence\ttau+\mu$, multiplying with
$(\beta-1)$ and adding it to~\eqref{eq:lsq:VI} we see that the solution $\uu\in K^s$ satisfies the variational
inequality
\begin{align}\tag{VIa}\label{eq:VI:a}
  \blfa(\uu,\vv-\uu) \geq F_\beta(\vv-\uu) \quad\text{for all }\vv\in K^s.
\end{align}

For the derivation of our second variational inequality 
let $\uu=(u,\ssigma,\lambda)\in K^{0}$ denote the unique solution of~\eqref{eq:model:fo} with $f\in L^2(\Omega)$, $g\in
H^1(\Omega)$, $g|_\Gamma \leq 0$.
Recall that $\lambda = -\Delta u-f$. By~\eqref{eq:model:VI} we have that 
\begin{align*}
  \dual{\lambda}{v-u} = \ip{\nabla u}{\nabla(v-u)} - \ip{f}{v-u} \geq 0 
\end{align*}
for all $v\in H_0^1(\Omega)$, $v\geq g$.
Thus, $\uu\in K^0$ satisfies the variational inequality
\begin{align}\tag{VIb}\label{eq:VI:b}
  \blfb(\uu,\vv-\uu) \geq G_\beta(\vv-\uu) \quad\text{for all }\vv\in K^0.
\end{align}

Our final method is based on the observation that for $\mu\geq 0$, we have that $\dual{\mu}{u-g}\geq 0$  for $u\geq g\in
H_0^1(\Omega)$. Together with the compatibility $\dual{\lambda}{u-g}=0$ we conclude $\dual{\mu-\lambda}{u-g}\geq 0$.
Thus, $\uu\in K^1$ satisfies the variational inequality
\begin{align}\tag{VIc}\label{eq:VI:c}
  \blfc(\uu,\vv-\uu) \geq H_\beta(\vv-\uu) \quad\text{for all }\vv\in K^1.
\end{align}

Note that $\blfa$ is symmetric, whereas $\blfb$, $\blfc$ are not.

\subsection{Solvability}
In what follows we analyse the (unique) solvability of the variational inequalities~\eqref{eq:VI:a}--\eqref{eq:VI:c}
in a uniform manner (including discretizations).
\begin{lemma}\label{lem:blf}
  Suppose $\beta>0$.
  Let $A\in \{\blfa,\blfb,\blfc\}$. There exists $C_\beta>0$ depending only on $\beta>0$ and $\Omega$
  such that
  \begin{align*}
    |A(\uu,\vv)| \leq C_\beta \norm{\uu}U\norm{\vv}U \quad\text{for all }\uu,\vv\in U.
  \end{align*}

  If $\beta\geq 1+\CF^2$, then $A$ is coercive, i.e.,
  \begin{align*}
    C \norm{\uu}U^2 \leq A(\uu,\uu) \quad\text{for all }\uu\in U.
  \end{align*}
  The constant $C>0$ is independent of $\beta$ and $\Omega$.
\end{lemma}
\begin{proof}
  We prove boundedness of $A = \blfa$. Let $\uu=(u,\ssigma,\lambda),\vv=(v,\ttau,\mu)\in U$ be given.
  The Cauchy-Schwarz inequality together with the Friedrichs' inequality and boundedness of the divergence operator
  yields
  \begin{align*}
    |\blfa(\uu,\vv)| &\leq \beta \norm{\divergence\ssigma+\lambda}{}\norm{\divergence\ttau+\mu}{}
    +\norm{\nabla u-\ssigma}{}\norm{\nabla v-\ttau}{} \\
    &\quad+ \tfrac12(\dual{\divergence\ttau+\mu}{u} -\dual{\divergence\ttau}u +
    \dual{\divergence\ssigma+\lambda}v-\dual{\divergence\ssigma}v) \\
    &\leq \beta \norm{\divergence\ssigma+\lambda}{}\norm{\divergence\ttau+\mu}{}
    +\norm{\nabla u-\ssigma}{}\norm{\nabla v-\ttau}{} \\
    &\quad+ \tfrac12\big( (\CF\norm{\divergence\ttau+\mu}{}+\norm{\ttau}{})\norm{\nabla u}{} +
    (\CF\norm{\divergence\ssigma+\lambda}{}+\norm{\ssigma}{})\norm{\nabla v}{}\big).
  \end{align*}
  This shows boundedness of $\blfa(\cdot,\cdot)$. Similarly, one concludes boundedness of $\blfb(\cdot,\cdot)$ and
  $\blfc(\cdot,\cdot)$.
  
  For the proof of coercivity, observe that $\blfa(\ww,\ww) = \blfb(\ww,\ww) = \blfc(\ww,\ww)$ for all $\ww\in U$.
  We stress that coercivity directly follows from the arguments given in the proof of~\cref{thm:lsq}.
  Note that the choice of $\beta$ yields
  \begin{align*}
    A(\ww,\ww) &\geq (1+\CF^2) \norm{\divergence\cchi+\nu}{}^2 + \norm{\nabla w-\cchi}{}^2
    + \dual{\nu}{w} 
  \end{align*}
  for $\ww=(w,\cchi,\nu)\in U$.
  The right-hand side can be further estimated following the argumentation as in the proof of~\cref{thm:lsq} which gives us
  \begin{align*}
    (1+\CF^2) \norm{\divergence\cchi+\nu}{}^2 + \norm{\nabla w-\cchi}{}^2 + \dual{\nu}w \gtrsim \norm{\ww}U^2.
  \end{align*}
  This finishes the proof.
\end{proof}

\begin{remark}\label{rem:scale}
  Recall that $\CF\leq \diam(\Omega)$. Therefore, we can always choose $\beta=1+\diam(\Omega)^2$ to ensure coercivity of
  our bilinear forms.
  Note that a scaling of $\Omega$ such that $\diam(\Omega)\leq 1$ implies that we can choose $\beta=2$. 
  Furthermore, observe that a scaling of $\Omega$ transforms~\eqref{eq:model} to an equivalent obstacle problem (with
  appropriate redefined functions $f,g$).
  To be more precise, define $\widetilde u(x) := u(dx)$ with $d:=\diam(\Omega)>0$ and $u\in H_0^1(\Omega)$ the
  solution of~\eqref{eq:model}. Moreover, set $\widetilde f(x) = d^2 f(dx)$, $\widetilde g(x) := g(dx)$.
  Then, $\widetilde u$ solves~\eqref{eq:model} in $\widetilde\Omega := \set{x/d}{x\in\Omega}$ with $f,g$ replaced by $\widetilde
  f,\widetilde g$.
\end{remark}

The variational inequalities~\eqref{eq:VI:a}--\eqref{eq:VI:c} are of the first kind and
we use a standard framework for the analysis (Lions-Stampacchia theorem), see~\cite{Glowinski08,Glowinski81,KinderlehrerStampacchia}.
\begin{theorem}\label{thm:solvability}
  Suppose $\beta\geq 1+\CF^2$. Let $A\in \{\blfa,\blfb,\blfc\}$ and let $F: U\to\R$ denote a bounded linear
  functional. 
  If $K\subseteq U$ is a non-empty convex and closed subset, then the variational inequality
  \begin{align}\label{eq:VI:abstract}
    \text{Find }\uu\in K \text{ s.t. } A(\uu,\vv-\uu) \geq F(\vv-\uu) \quad\text{for all }\vv\in K
  \end{align}
  admits a unique solution.

  In particular, for $f\in L^2(\Omega)$, $g\in H_0^1(\Omega)$ each of the problems~\eqref{eq:VI:a},~\eqref{eq:VI:b},~\eqref{eq:VI:c} has a unique solution and the
  problems are
  equivalent to~\eqref{eq:model:fo}.
\end{theorem}
\begin{proof}
  By the assumption on $\beta$,~\cref{lem:blf} proves that the bilinear forms are coercive and bounded.
  Then, unique solvability of~\eqref{eq:VI:abstract} follows from the Lions-Stampacchia theorem, 
  see, e.g.,~\cite{Glowinski08,Glowinski81,KinderlehrerStampacchia}.

  Unique solvability of~\eqref{eq:VI:a}--\eqref{eq:VI:c} follows since the functionals $F_\beta$, $G_\beta$, $H_\beta$
  are linear and bounded.
  Boundedness of $F_\beta$ can be seen from
  \begin{align*}
    |F_\beta(\vv)| &= |-\beta\ip{f}{\divergence\ttau+\mu}{} + \tfrac12\ip{\divergence\ttau+\mu}g -
    \tfrac12\dual{\divergence\ttau}g|  \\
    &\leq \beta \norm{f}{}\norm{\divergence\ttau+\mu}{} + \tfrac12\norm{\divergence\ttau+\mu}{}\norm{g}{} 
    + \tfrac12\norm{\ttau}{}\norm{\nabla g}{} \lesssim (\norm{f}{} +\norm{\nabla g}{})\norm{\vv}U.
  \end{align*}
  The same arguments prove that $G_\beta$ and $H_\beta$ are bounded.

  Finally, equivalence to~\eqref{eq:model:fo} follows since all problems admit unique solutions and by construction the
  solution of~\eqref{eq:model:fo} also solves each of the problems~\eqref{eq:VI:a}--\eqref{eq:VI:c}.
\end{proof}

\begin{remark}
  We stress that the assumption $g\in H_0^1(\Omega)$ is necessary. If $g\in H^1(\Omega)$ then 
  the term $\dual{\mu}g$ in $F_\beta$, $H_\beta$ is not well-defined.
  However, this term does not appear in $G_\beta$ and therefore the variational inequality in~\eqref{eq:VI:b} admits a
  unique solution if we only assume $g\in H^1(\Omega)$ with $g|_\Gamma\leq 0$.
\end{remark}

\begin{remark}
  The variational inequality~\eqref{eq:VI:a} corresponds to a least-squares finite element method with convex functional
  \begin{align*}
    J_\beta(\uu;f,g) := \blfa(\uu,\uu)-2F_\beta(\uu)+\beta\ip{f}f.
  \end{align*}
  Then,~\cref{thm:solvability} proves that the problem
  \begin{align*}
    J_\beta(\uu;f,g) = \min_{\vv\in K} J_\beta(\vv;f,g)
  \end{align*}
  admits a unique solution for all non-empty convex and closed sets $K\subseteq U$. 
  Moreover, $J_\beta(\uu;f,g)\simeq J(\uu;f,g)$ for $\uu\in K^s$, so that this problem is equivalent to~\eqref{eq:lsq}
  for $K=K^s$.
\end{remark}

\subsection{A priori analysis}\label{sec:apriori}
The following three results provide general bounds on the approximation error.
The proofs are based on standard arguments, see, e.g.,~\cite{Falk74}. 
We give details for the proof of the first result, the
others follow the same lines of argumentation and are left to the reader.

\begin{theorem}\label{thm:VI:a}
  Suppose $\beta\geq 1+\CF^2$. 
  Let $\uu\in K^s$ denote the solution of~\eqref{eq:VI:a}, where $f\in L^2(\Omega)$, $g\in H_0^1(\Omega)$.
  Let $K_h\subset U$ denote a non-empty convex and closed subset and let $\uu_h \in K_h$ denote the solution
  of~\eqref{eq:VI:abstract} with $A=\blfa$, $F=F_\beta$ and $K=K_h$. 
  It holds that
  \begin{align*}
    \norm{\uu-\uu_h}U^2 &\leq \Copt \Big( \inf_{\vv_h\in K_h} \big( \norm{\uu-\vv_h}U^2 + 
    |\dual{\lambda}{v_h-u}+\dual{\mu_h-\lambda}{u-g}| \big)
    \\
    &\qquad\qquad + \inf_{\vv\in K^s} | \dual{\lambda}{v-u_h} + \dual{\mu-\lambda_h}{u-g}|\Big).
  \end{align*}
  The constant $\Copt>0$ depends only on $\beta$ and $\Omega$.
\end{theorem}
\begin{proof}
  Throughout let $\vv=(v,\ttau,\mu)\in K^s$, $\vv_h=(v_h,\ttau_h,\mu_h)\in K_h$ and let 
  $\uu=(u,\ssigma,\lambda)\in K^s$ denote the exact solution of~\eqref{eq:VI:a}. Thus, $\divergence\ssigma+\lambda+f=0$
  and $\nabla u-\ssigma = 0$.
  For arbitrary $\ww = (w,\cchi,\nu)\in U$ it holds that
  \begin{align}
    \begin{split}\label{eq:VI:a:id}
    \blfa(\uu,\ww) &= \beta\ip{\divergence\ssigma+\lambda}{\divergence\cchi+\nu} 
    + \ip{\nabla u-\ssigma}{\nabla w-\cchi} + \tfrac12(\dual{\lambda}w+\dual\nu{u}) \\
    &= -\beta\ip{f}{\divergence\cchi+\nu} +\tfrac12\dual\nu{g} 
    +\tfrac12(\dual{\lambda}w + \dual\nu{u-g}) \\
    &= F_\beta(\ww) + \tfrac12(\dual{\lambda}w + \dual\nu{u-g}).
    \end{split}
  \end{align}
  Using coercivity of $\blfa(\cdot,\cdot)$, identity~\eqref{eq:VI:a:id} and the fact that $\uu_h$ solves the discretized 
  variational inequality (on $K_h$) shows that
  \begin{align*}
    \norm{\uu-\uu_h}U^2 &\lesssim \blfa(\uu-\uu_h,\uu-\uu_h) \\& = \blfa(\uu,\uu-\uu_h) - \blfa(\uu_h,\uu-\vv_h)
    -\blfa(\uu_h,\vv_h-\uu_h) \\
    &\leq F_\beta(\uu-\uu_h) + \tfrac12(\dual{\lambda}{u-u_h} + \dual{\lambda-\lambda_h}{u-g})
    \\&\qquad- \blfa(\uu_h,\uu-\vv_h) - F_\beta(\vv_h-\uu_h)
    \\&= F_\beta(\uu-\vv_h) + \tfrac12(\dual{\lambda}{u-u_h} + \dual{\lambda-\lambda_h}{u-g}) 
    - \blfa(\uu_h,\uu-\vv_h)
  \end{align*}
  Note that $0=\dual{\lambda}{u-g}\leq \dual{\lambda}{v-g}$ and $\dual{\lambda}{u-g}\leq \dual{\mu}{u-g}$. Hence,
  \begin{align*}
    \dual{\lambda}{u-u_h}+\dual{\lambda-\lambda_h}{u-g} &
    = \dual{\lambda}{u-g+g-u_h} + \dual{\lambda-\lambda_h}{u-g}
    \\
    &\leq \dual{\lambda}{v-g+g-u_h} + \dual{\mu-\lambda_h}{u-g}.
  \end{align*}
  This and identity~\eqref{eq:VI:a:id} with $\ww =\uu-\vv_h$ imply that
  \begin{align*}
    &F_\beta(\uu-\vv_h)-\blfa(\uu_h,\uu-\vv_h) + \tfrac12(\dual{\lambda}{u-u_h} + \dual{\lambda-\lambda_h}{u-g})
    \\
    &\quad\leq \blfa(\uu-\uu_h,\uu-\vv_h) - \tfrac12(\dual{\lambda}{u-v_h}+\dual{\lambda-\mu_h}{u-g}) 
    \\ &\qquad + \tfrac12(\dual{\lambda}{v-u_h} + \dual{\mu-\lambda_h}{u-g}).
  \end{align*}
  Putting altogether, boundedness of $\blfa(\cdot,\cdot)$ and an application of Young's inequality with parameter
  $\delta>0$ show that
  \begin{align*}
    \norm{\uu-\uu_h}U^2 &\lesssim \frac{\delta}2 \norm{\uu-\uu_h}U^2 + \frac{\delta^{-1}}2\norm{\uu-\vv_h}U^2
    + |\dual{\lambda}{v_h-u}+\dual{\mu_h-\lambda}{u-g}|
    \\ &\qquad
    + |\dual{\lambda}{v-u_h} + \dual{\mu-\lambda_h}{u-g}|.
  \end{align*}
  Subtracting the term $\delta/2\norm{\uu-\uu_h}U^2$ for some sufficiently $\delta>0$ finishes the proof since $\vv\in
  K^s$, $\vv_h\in K_h$ are arbitrary.
\end{proof}

\begin{theorem}\label{thm:VI:b}
  Suppose $\beta\geq 1+\CF^2$.
  Let $\uu\in K^0$ denote the solution of~\eqref{eq:VI:b}, 
  where $f\in L^2(\Omega)$, $g\in H^1(\Omega)$ with $g|_\Gamma\leq 0$.
  Let $K_h\subset U$ denote a non-empty convex and closed subset and let $\uu_h \in K_h$ denote the solution
  of~\eqref{eq:VI:abstract} with $A=\blfb$, $F=G_\beta$, and $K=K_h$. 
  It holds that
  \begin{align*}
    \norm{\uu-\uu_h}U^2 &\leq \Copt \Big( \inf_{\vv_h\in K_h}\big(\norm{\uu-\vv_h}U^2 + 
    |\dual{\lambda}{v_h-u}|\big)
     + \inf_{\vv\in K^0} | \dual{\lambda}{v-u_h}|\Big).
  \end{align*}
  The constant $\Copt>0$ depends only on $\beta$ and $\Omega$.
\end{theorem}

\begin{theorem}\label{thm:VI:c}
  Suppose $\beta\geq 1+\CF^2$.
  Let $\uu\in K^1$ denote the solution of~\eqref{eq:VI:c}, where $f\in L^2(\Omega)$, $g\in H_0^1(\Omega)$.
  Let $K_h\subset U$ denote a non-empty convex and closed subset and let $\uu_h \in K_h$ denote the solution
  of~\eqref{eq:VI:abstract} with $A=\blfc$, $F=H_\beta$, and $K=K_h$. 
  It holds that
  \begin{align*}
    \norm{\uu-\uu_h}U^2 &\leq \Copt \Big( \inf_{\vv_h\in K_h}\big(\norm{\uu-\vv_h}U^2 + 
    |\dual{\mu_h\!-\!\lambda}{u\!-\!g}|\big)
     + \inf_{\vv\in K^1} | \dual{\mu\!-\!\lambda_h}{u\!-\!g}|\Big).
  \end{align*}
  The constant $\Copt>0$ depends only on $\beta$ and $\Omega$.
\end{theorem}

\subsection{Discretization}
Let $\TT$ denote a regular triangulation of 
$\Omega$, $\bigcup_{T\in\TT} \overline{T}=\overline\Omega$. We assume that $\TT$ is $\kappa$-shape regular, i.e.,
\begin{align*}
  \sup_{T\in\TT} \frac{\diam(T)^n}{|T|} \leq \kappa < \infty.
\end{align*}
Moreover, let $\NN$ denote the nodes of the mesh $\TT$ and $h_\TT\in L^\infty(\Omega)$ the mesh-size function, $h_\TT|_T
:= h_T := \diam(T)$ for $T\in\TT$. Set $h:=\max_{T\in\TT}\diam(T)$.
We use standard finite element spaces for the discretization. Let $\PP^p(\TT)$ denote the space of $\TT$-elementwise
polynomials of degree less or equal than $p\in\N_0$.
Let $\RT^p(\TT)$ denote the Raviart-Thomas space of degree $p\in\N_0$, $\cS_0^{p+1}(\TT) := \PP^{p+1}(\TT)\cap
H_0^1(\Omega)$, and
\begin{align*}
  U_{hp} := \cS_0^{p+1}(\TT) \times \RT^p(\TT) \times \PP^p(\TT).
\end{align*}
Clearly, $U_{hp} \subset U$.
We stress that the polynomial degree is chosen, so that the best approximation in the norm $\norm{\cdot}U$ is of order
$h^{p+1}$.

To define admissible convex sets for the discrete variational inequalities we need to put constraints on functions from 
the space $\cS_0^{p+1}(\TT)$ or from $\PP^p(\TT)$ or both.
Let us remark that for a polynomial degree $\geq 2$ such constraints are not straightforward to implement.
One possibility would be to impose such constraints pointwise and then
analyse the consistency error (this can be done with the results from~\cref{sec:apriori}).
For some $hp$-FEM method for elliptic obstacle problems we refer to~\cite{BanzSchroeder,BanzStephan}.
In order to avoid such quite technical treatments and for a simpler representation of the basic ideas
we consider from now on the lowest-order case only, where the linear constraints can easily be built in.
To that end define the non-empty convex subsets
\begin{align}
  K_h^s &:= \set{(v_h,\ttau_h,\mu_h)\in U_{h0}}{\mu_h\geq0, \, v_h(x)\geq g(x) \text{ for all } x\in\NN}, \label{eq:dCone:a}\\
  K_h^0 &:= \set{(v_h,\ttau_h,\mu_h)\in U_{h0}}{v_h(x)\geq g(x) \text{ for all } x\in\NN}, \label{eq:dCone:b}\\
  K_h^1 &:= \set{(v_h,\ttau_h,\mu_h)\in U_{h0}}{\mu_h\geq0}. \label{eq:dCone:c}
\end{align}
In the definition of $K_h^s$, $K_h^0$ we assume $g\in H^1(\Omega)\cap C^0(\overline\Omega)$ so that the
point evaluation is well-defined.

For the analysis of the convergence rates we use the nodal interpolation operator $I_h : H^2(\Omega)\to
\cS^1(\TT):= \PP^1(\TT)\cap C^0(\overline\Omega)$,
the Raviart-Thomas projector $\Pidiv_h : H^1(\Omega)^n \to \RT^0(\TT)$, and the $L^2(\Omega)$ projector
$\Pi_h : L^2(\Omega)\to \PP^0(\TT)$.
Observe that with $v\geq 0$, $\mu\geq0$ we have (with sufficient regularity) that $I_h v\geq 0$, $\Pi_h\mu\geq 0$. 
Moreover, recall the commutativity property $\divergence\Pidiv_h = \Pi_h\divergence$, as well as the 
approximation properties
\begin{align}
  \norm{v-I_hv}{} + h\norm{\nabla(v-I_hv)}{} &\lesssim h^2\norm{D^2v}{}, \\
  \norm{\ttau-\Pidiv_h\ttau}{} &\lesssim h \norm{\nabla \ttau}{}, \\
  \norm{\mu-\Pi_h\mu}{} &\lesssim \norm{h_\TT\pwnabla\mu}{}.
\end{align}
Here, $\nabla\ttau$ is understood componentwise, $\pwnabla\mu$ denotes the $\TT$-elementwise gradient of $\mu\in
H^1(\TT) := \set{\nu\in L^2(\Omega)}{\nu|_T \in H^1(T), \, T\in\TT}$. 
Set $\norm{\nu}{H^1(\TT)}^2 := \norm\nu{}^2 + \norm{\pwnabla\nu}{}^2$.
The involved constants depend only on the $\kappa$-shape regularity of $\TT$ but are otherwise independent of $\TT$.
Furthermore, for $\mu\in L^2(\Omega)$, it also holds that
\begin{align*}
  \norm{\mu-\Pi_h\mu}{-1} \lesssim \norm{h_\TT(\mu-\Pi_h\mu)}{},
\end{align*}
which follows from the definition of the dual norm, the projection and approximation property of $\Pi_h$.

\begin{theorem}\label{thm:VI:a:conv}
  Suppose $\beta\geq 1+\CF^2$. 
  Let $\uu\in K^s$ denote the solution of~\eqref{eq:VI:a} with data $f\in L^2(\Omega)$, $g\in H_0^1(\Omega)$.
  Let $K_h^s$ denote the set defined in~\eqref{eq:dCone:a} and let $\uu_h \in K_h^s$ denote the solution
  of~\eqref{eq:VI:abstract} with $A=\blfa$, $F=F_\beta$, and $K=K_h^s$. 
  If $u \in H^2(\Omega)$, $\lambda\in H^1(\TT)$, $g\in H^2(\Omega)$ and $f\in H^1(\TT)$, then
  \begin{align*}
    \norm{\uu-\uu_h}U &\leq \Capp h (\norm{u}{H^2(\Omega)} + \norm{\pwnabla f}{} +\norm{\lambda}{H^1(\TT)} + \norm{g}{H^2(\Omega)}).
  \end{align*}
  The constant $\Capp>0$ depends only on $\beta$, $\Omega$, and $\kappa$-shape regularity of $\TT$.
\end{theorem}
\begin{proof}
  Choose $\vv_h = (I_hu,\Pidiv_h\ssigma,\Pi_h\lambda)\in K_h^s$. The commutativity property of $\Pidiv_h$ shows that
  \begin{align*}
    \divergence(\ssigma-\Pidiv_h\ssigma)+\lambda-\Pi_h\lambda = (1-\Pi_h)(\divergence\ssigma+\lambda) = (1-\Pi_h)f.
  \end{align*}
  Therefore, using the approximation properties of the involved operators proves
  \begin{align*}
    \norm{\uu-\vv_h}U \leq \norm{(1-\Pi_h)f}{} + \norm{\ssigma-\Pidiv_h\ssigma}{} + \norm{\nabla(u-I_h u)}{}
    \lesssim h\norm{\pwnabla f}{} + h\norm{u}{H^2(\Omega)}.
  \end{align*}
  Moreover, 
  \begin{align*}
    |\dual{\lambda}{I_h u-u}| \leq \norm{\lambda}{}h^2\norm{D^2u}{}\lesssim
    h^2(\norm{u}{H^2(\Omega)}^2+\norm{\lambda}{}^2)
  \end{align*}
  and 
  \begin{align*}
    |\dual{\Pi_h\lambda-\lambda}{u-g}| &\leq \norm{(1-\Pi_h)\lambda}{-1}\norm{\nabla(u-g)}{}
    \lesssim h^2 \norm{\pwnabla\lambda}{}\big(\norm{\nabla u}{}+\norm{\nabla g}{}\big).
  \end{align*}
  Summing up we have that
  \begin{align*}
    &\inf_{\vv_h\in K_h^s} \big( \norm{\uu-\vv_h}U^2 + 
    |\dual{\lambda}{v_h-u}+\dual{\mu_h-\lambda}{u-g}| \big)
    \\
    &\qquad\quad\lesssim h^2 \big(\norm{u}{H^2(\Omega)}^2 + \norm{\pwnabla f}{}^2 
    +\norm{\pwnabla\lambda}{}^2 + \norm{\nabla g}{}^2\big).
  \end{align*}
  Therefore, in view of~\cref{thm:VI:a} it only remains to estimate the consistency error
  \begin{align*}
    \inf_{\vv\in K^s} | \dual{\lambda}{v-u_h} + \dual{\mu-\lambda_h}{u-g}|.
  \end{align*}
  Define $\vv := (v,\cchi,\mu):=(v,0,\lambda_h)\in U$ with $v:=\sup\{u_h,g\}$ and observe that $\vv\in K^s$.
  This directly leads to $\dual{\mu-\lambda_h}{u-g} = 0$. For the remaining term we follow the seminal 
  work~\cite{Falk74} of Falk.
  The same lines as in the proof of~\cite[Lemma~4]{Falk74} show that
  \begin{align*}
    |\dual{\lambda}{v-u_h}|\leq \norm{\lambda}{} \norm{v-u_h}{} \leq 
    \norm{\lambda}{} \norm{g-I_hg}{} \lesssim
    h^2\norm{g}{H^2(\Omega)}\norm{\lambda}{}.
  \end{align*}
  This finishes the proof.
\end{proof}

The proof of the following result can be obtained in the same fashion as the previous one and is therefore omitted.
Note that in contrast to the last result the additional regularity assumption on 
the Lagrange multiplier $\lambda\in H^1(\TT)$ is not needed.

\begin{theorem}\label{thm:VI:b:conv}
  Suppose $\beta\geq 1+\CF^2$. 
  Let $\uu\in K^0$ denote the solution of~\eqref{eq:VI:b} with data $f\in L^2(\Omega)$, $g\in H^1(\Omega)$,
  $g|_\Gamma\leq 0$.
  Let $\uu_h \in K_h$ denote the solution
  of~\eqref{eq:VI:abstract} with $A=\blfb$, $F=G_\beta$, and $K=K_h$, where either $K_h=K_h^s$ or $K_h=K_h^0$.
  If $u \in H^2(\Omega)$, $g\in H^2(\Omega)$ and $f\in H^1(\TT)$, then
  \begin{align*}
    \norm{\uu-\uu_h}U &\leq \Capp h (\norm{u}{H^2(\Omega)} + \norm{\pwnabla f}{} +\norm{\lambda}{} + \norm{g}{H^2(\Omega)}).
  \end{align*}
  The constant $\Capp>0$ depends only on $\beta$, $\Omega$, and $\kappa$-shape regularity of $\TT$.
\end{theorem}

Finally, we show convergence rate for problem~\eqref{eq:VI:c} and its approximation.
Note that for the sets $K_h^1$, $K_h^s$ defined in~\eqref{eq:dCone:c},~\eqref{eq:dCone:a} 
it holds that $K_h^s\subset K_h^1\subset K^1$ and thus the consistency error, see~\cref{thm:VI:c}, vanishes.
Furthermore, note that we do not need additional regularity assumptions on the obstacle $g$.
The proof is similar to the one of~\cref{thm:VI:a:conv} and is therefore left to the reader.

\begin{theorem}\label{thm:VI:c:conv}
  Suppose $\beta\geq 1+\CF^2$. 
  Let $\uu\in K^1$ denote the solution of~\eqref{eq:VI:c} with data $f\in L^2(\Omega)$, $g\in H_0^1(\Omega)$.
  Let $\uu_h \in K_h$ denote the solution
  of~\eqref{eq:VI:abstract} with $A=\blfc$, $F=H_\beta$, and $K=K_h$, where either $K_h=K_h^s$ or $K_h=K_h^1$.
  If $u \in H^2(\Omega)$, $\lambda\in H^1(\TT)$ and $f\in H^1(\TT)$, then
  \begin{align*}
    \norm{\uu-\uu_h}U &\leq \Capp h (\norm{u}{H^2(\Omega)} + \norm{\pwnabla f}{} +\norm{\pwnabla\lambda}{} +
    \norm{g}{H^1(\Omega)}).
  \end{align*}
  The constant $\Capp>0$ depends only on $\beta$, $\Omega$, and $\kappa$-shape regularity of $\TT$.
\end{theorem}

To shortly summarize this section, we have defined and analyzed three different variational inequalities and its
discrete variants. The following table shows which discrete sets can be used for approximating solutions
with~\eqref{eq:VI:a}--\eqref{eq:VI:c} and which assumptions we need for the obstacle so that the formulation is
well-defined.
\begin{table}[h]
\centering
\begin{tabular}{|c|c|c|}
\hline
& Convex set & Obstacle \\ \hline\hline
\eqref{eq:VI:a} & $K_h^s$ & $g\in H_0^1(\Omega)\cap C^0(\overline\Omega)$ \\ \hline
\eqref{eq:VI:b} & $K_h^0$, $K_h^s$ & $g\in H^1(\Omega)\cap C^0(\overline\Omega)$, $g|_\Gamma\leq 0$ \\ \hline
\eqref{eq:VI:c} & $K_h^1$ & $g\in H_0^1(\Omega)$ \\ \hline
\eqref{eq:VI:c} & $K_h^s$ & $g\in H_0^1(\Omega)\cap C^0(\overline\Omega)$ \\ \hline
\end{tabular}
\caption{Overview on which convex sets can be used for the discrete versions of the variational
  inequalities~\eqref{eq:VI:a}--\eqref{eq:VI:c} and corresponding assumptions on the obstacle function.} \label{tab:ass}
\end{table}

\section{A posteriori analysis}\label{sec:apost}
In this section we derive reliable error bounds that can be used as a posteriori estimators.
We define
\begin{align*}
  \oscf &:= \oscf(f) := \norm{(1-\Pi_h)f}{}.
\end{align*}
The estimator below includes the residual term
\begin{align*}
  \eta^2:= \eta(\uu_h,f)^2:= \norm{\divergence\ssigma_h+\lambda_h+\Pi_hf}{}^2 
  + \norm{\nabla u_h-\ssigma_h}{}^2,
\end{align*}
which can be localized.
The derivation of our estimators is quite simple and is based on the following observation.
Let $\uu\in K^s\subset K^j$ denote the unique solution of~\eqref{eq:model:fo} and let $\uu_h\in U_{h0}$ be arbitrary.
Take $\beta = 1+\CF^2$ and 
recall that by~\cref{lem:blf} it holds that $\blfa(\vv,\vv)=\blfb(\vv,\vv)=\blfc(\vv,\vv)\gtrsim\norm{\vv}U^2$ for all
$\vv\in U$.
Then, together with the Pythagoras theorem $\norm{\mu}{}^2 = \norm{(1-\Pi_h)\mu}{}^2 + \norm{\Pi_h\mu}{}^2$ for $\mu\in L^2(\Omega)$
and using $\divergence\ssigma+\lambda+f=0$, $\nabla u = \ssigma$,
$\divergence\ssigma_h+\lambda_h\in \PP^0(\TT)$, it follows that
\begin{align}\label{eq:apost:general}
\begin{split}
  \norm{\uu-\uu_h}U^2 &\lesssim \beta \norm{\divergence\ssigma_h+\lambda_h+f}{}^2 + \norm{\nabla u_h-\ssigma_h}{}^2 
  + \dual{\lambda_h-\lambda}{u_h-u}
  \\
  &= \beta\norm{\divergence\ssigma_h+\lambda_h+\Pi_h f}{}^2 + \beta\oscf^2 + \norm{\nabla u_h-\ssigma_h}{}^2 
  + \dual{\lambda_h-\lambda}{u_h-u}
  \\
  &\leq \beta( \eta^2 + \oscf^2) + \dual{\lambda_h-\lambda}{u_h-u}.
\end{split}
\end{align}

The remaining results in this section are proved by estimating the duality term $\dual{\lambda_h-\lambda}{u_h-u}$
from~\eqref{eq:apost:general}.
In particular, the proof of the next result employs only $\lambda_h\geq 0$ 
We will need the positive resp. negative part of a function $v : \Omega \to\R$,
\begin{align*}
  v_{+} := \max\{0,v\}, \quad v_{-} := -\min\{0,v\}.
\end{align*}
This definition implies that $v = v_+-v_-$.
The ideas of estimating the duality term are similar as in~\cite{GustafssonStenbergVideman_MixedFEMObstacle,Veeser01I}
and references therein, 
see also~\cite{DPGsignorini} for a related estimate for Signorini-type problems.
Note that 
we do not need to assume $g\in H_0^1(\Omega)$.

\begin{theorem}\label{thm:apost:a}
  Let $\uu\in K^s$ denote the solution of~\eqref{eq:model:fo}.
  Let $\uu_h \in K_h$, where $K_h\in\{K_h^s,K_h^1\}$, be arbitrary. The error satisfies
  \begin{align*}
    \norm{\uu-\uu_h}U^2 \leq \Crel \big( \eta^2 + \rho^2 + \oscf^2 \big),
  \end{align*}
  where the estimator contribution $\rho$ is given by
  \begin{align*}
    \rho^2 := \dual{\lambda_h}{(u_h-g)_+} + \norm{\nabla(g-u_h)_+}{}^2.
  \end{align*}
  The constant $\Crel>0$ depends only on $\Omega$.
\end{theorem}
\begin{proof}
  In view of estimate~\eqref{eq:apost:general} we only have to tackle the term $\dual{\lambda_h-\lambda}{u_h-u}$.
  Define $v_h := \max\{u_h,g\}$. Clearly, $v_h\geq g$ and $v_h\in H_0^1(\Omega)$. Note that $\lambda = -\Delta u - f \in H^{-1}(\Omega)$.
  Therefore, $\dual{\lambda}v = \ip{\nabla u}{\nabla v}-\ip{f}v$ for all $v\in H_0^1(\Omega)$ and using the variational
  inequality for the exact solution~\eqref{eq:model:VI} yields
  \begin{align*}
    -\dual{\lambda}{u_h-u} &= -\dual{\lambda}{u_h-v_h}-\dual{\lambda}{v_h-u} \leq -\dual{\lambda}{u_h-v_h} \\
    &= \dual{\lambda}{(u_h-g)_-} = \dual{\lambda-\lambda_h}{(u_h-g)_-} + \dual{\lambda_h}{(u_h-g)_-}
    \\
    &\leq \frac\delta2\norm{\lambda-\lambda_h}{-1}^2 + \frac{\delta^{-1}}2\norm{\nabla(u_h-g)_-}{}^2 +
    \dual{\lambda_h}{(u_h-g)_-}
  \end{align*}
  for all $\delta>0$.
  Employing $\lambda_h\geq 0$, $g-u\leq 0$, and $v+v_-=v_+$ we further infer that
  \begin{align*}
    \dual{\lambda_h-\lambda}{u_h-u} &\leq \dual{\lambda_h}{u_h-g+(u_h-g)_-} + \dual{\lambda_h}{g-u} 
    \\ &\qquad + \frac\delta2\norm{\lambda-\lambda_h}{-1}^2 + \frac{\delta^{-1}}2\norm{\nabla(u_h-g)_-}{}^2
    \\
    &\leq \dual{\lambda_h}{(u_h-g)_+} + \frac\delta2\norm{\lambda-\lambda_h}{-1}^2 +
    \frac{\delta^{-1}}2\norm{\nabla(u_h-g)_-}{}^2.
  \end{align*}
  Recall that $\norm{\lambda-\lambda_h}{-1}\leq \norm{\uu-\uu_h}V\lesssim \norm{\uu-\uu_h}U$, where the involved
  constant depends only on $\Omega$.
  Thus, choosing $\delta>0$ sufficiently small the proof is concluded with~\eqref{eq:apost:general}.
\end{proof}

We could derive a similar estimate if $\uu_h\in K_h^0$ by changing the role of $u_h$ and $\lambda_h$ resp. $u$ and
$\lambda$ in the proof. 
However, this leads to an estimator with a non-local term.
To see this, suppose $g=0$. Then, following the last proof we get
\begin{align*}
  \dual{\lambda_h-\lambda}{u_h-u} \leq \dual{(\lambda_h)_+}{u_h} + \frac\delta2 \norm{\nabla(u-u_h)}{}^2 
  + \frac{\delta^{-1}}2 \norm{(\lambda_h)_{-}}{-1}^2
\end{align*}
for $\delta>0$. For the total error this would yield
\begin{align*}
  \norm{\uu-\uu_h}U^2 \lesssim \eta^2 + \oscf^2 + \dual{(\lambda_h)_+}{u_h} + \norm{(\lambda_h)_{-}}{-1}^2.
\end{align*}
The last term is not localizable and therefore it is not feasible to use this estimate as an a posteriori error
estimator in an adaptive algorithm. 

\begin{remark}
  The derived estimator is efficient up to the term $\rho$, i.e.,
  \begin{align*}
    \eta^2 + \oscf^2 \lesssim \norm{\uu-\uu_h}U^2.
  \end{align*}
  To see this, we employ the Pythagoras theorem to obtain 
  \begin{align*}
    \eta^2+\oscf^2 = \norm{\divergence\ssigma_h+\lambda_h+f}{}^2 + \norm{\nabla u_h-\ssigma_h}{}^2.
  \end{align*}
  Then, $\divergence\ssigma+\lambda=-f$, $\nabla u = \ssigma$ and the triangle inequality prove the asserted
  estimate.
  The proof of the efficiency
  estimate $\rho\lesssim \norm{\uu-\uu_h}U$ (up to possible data resp. obstacle oscillations) is an open problem.
\end{remark}

\section{Examples}
\label{sec:examples}

In this section we present numerical studies that demonstrate the performance of our proposed methods in different situations:
\begin{itemize}
  \item In~\cref{ex:smooth} we consider a problem on the unit square with smooth obstacle and known smooth solution.
  \item In~\cref{ex:Lshape} we consider the example from~\cite[Section~5.2]{CarstensenBartels04} where the solution 
    is known and exhibits a singularity.
  \item In~\cref{ex:pyramid} we consider a problem on an L-shaped domain with a pyramid-like obstacle and unknown
    solution. 
\end{itemize}
Before we come to a detailed discussion on the numerical studies some remarks are in order.
In all examples we choose $\beta = 1+\diam(\Omega)^2$ to ensure coercivity of the bilinear forms (\cref{lem:blf}).
This also implies that the Galerkin matrices associated to the bilinear forms $\blfa$, $\blfb$, and $\blfc$ are positive
definite.
Choosing standard basis functions for $\cS_0^1(\TT)$ (nodal basis), $\RT^0(\TT)$ (lowest-order Raviart-Thomas basis) and
$\PP^0(\TT)$ (characteristic functions), the constraints in the discrete
convex sets $K_h^\star$ are straightforward to impose.
The resulting discrete variational inequalities are then solved using a (primal-dual) active set strategy, 
see, e.g.,~\cite{MR1276702,KarkkainenKT_03_ALA}.

We define the error resp. total estimator by
\begin{align*}
  \err_U := \norm{\uu-\uu_h}U, \quad \est^2 := \eta^2 + \rho^2 + \oscf^2.
\end{align*}
Note that the estimator can be decomposed into local contributions,
\begin{align*}
  \est^2 = \sum_{T\in\TT} \est(T)^2 &=: \sum_{T\in\TT} \Big( \norm{\divergence\ssigma_h+\lambda_h+\Pi_h f}T^2
  + \norm{\nabla u_h-\ssigma_h}T^2 \\
  &\qquad\qquad+ \ip{\lambda_h}{(u_h-g)_+}_T + \norm{\nabla(g-u_h)_+}T^2
+ \norm{(1-\Pi_h)f}T^2 \Big),
\end{align*}
where $\norm\cdot{T}$ denotes the $L^2(T)$ norm and $\ip\cdot\cdot_T$ the $L^2(T)$ inner product.
Moreover, we will estimate the error in the weaker norm $\norm\cdot{V}$.
To do so we consider an upper bound given by
\begin{align*}
  \err_V^2 := \err_V(\uu)^2 := \norm{\nabla(u-u_h)}{}^2 + \norm{\ssigma-\ssigma_h}{}^2 
  + \norm{\lambda-\lambda_h}{-1,h}^2,
\end{align*}
where the evaluation of $\norm{\cdot}{-1,h}$ is based on the discrete $H^{-1}(\Omega)$ norm 
discussed in the seminal work~\cite{BPLdiscrete}:
Let $Q_h : L^2(\Omega)\to \cS_0^1(\TT)$ denote the $L^2(\Omega)$ projector.
Let $\mu\in L^2(\Omega)$.
We stress that using the projection and local approximation property of $Q_h$ yields
\begin{align*}
  \norm{(1-Q_h)\mu}{-1} = \sup_{0\neq v\in H_0^1(\Omega)} \frac{\dual{(1-Q_h)\mu}{(1-Q_h)v}}{\norm{\nabla v}{}}
  \lesssim \norm{h_\TT\mu}{},
\end{align*}
where the involved constant depends on shape regularity of $\TT$.
Following~\cite{BPLdiscrete} it holds that
\begin{align*}
  \norm{\mu}{-1} \leq \norm{(1-Q_h)\mu}{-1} + \norm{Q_h\mu}{-1} \lesssim 
  \norm{h_\cS \mu}{} + \norm{\nabla u_h[\mu]}{}
\end{align*}
where $u_h[\mu]\in \cS_0^1(\TT)$ is the solution of 
\begin{align*}
  \ip{\nabla u_h[\mu]}{\nabla v_h} = \dual{\mu}{v_h} \quad\text{for all }v_h\in \cS_0^1(\TT).
\end{align*}
Note that $\norm{\nabla u_h[\mu]}{}\leq \norm{\mu}{-1}$.
The estimate $\norm{Q_h\mu}{-1}\lesssim \norm{\nabla u_h[\mu]}{}$ 
depends on the stability of the projection $Q_h$ in $H^1(\Omega)$, $\norm{\nabla Q_h
v}{} \lesssim \norm{\nabla v}{}$ for $v\in H_0^1(\Omega)$, i.e.,
\begin{align*}
  \norm{Q_h\mu}{-1} &= \sup_{0\neq v\in H_0^1(\Omega)} \frac{\dual{Q_h \mu}v}{\norm{\nabla v}{}}
  = \sup_{0\neq v\in H_0^1(\Omega)} \frac{\dual{\mu}{Q_hv}}{\norm{\nabla v}{}}
  = \sup_{0\neq v\in H_0^1(\Omega)} \frac{\ip{\nabla u_h[\mu]}{\nabla Q_hv}}{\norm{\nabla v}{}}
  \\
  &\lesssim \sup_{0\neq v\in H_0^1(\Omega)} \frac{\ip{\nabla u_h[\mu]}{\nabla Q_h v}}{\norm{\nabla Q_h v}{}}
  = \norm{\nabla u_h[\mu]}{}.
\end{align*}
Here, we use newest-vertex bisection~\cite{stevenson:NVB} as refinement strategy where stability of the 
$L^2(\Omega)$ projection is known~\cite{Karkulik2013}.

We use an adaptive algorithm that basically consists of iterating the four steps
\begin{align*}
  \boxed{\emph{SOLVE}} \to \boxed{\emph{ESTIMATE}} \to \boxed{\emph{MARK}} 
  \to \boxed{\emph{REFINE}},
\end{align*}
where the marking step is done with the bulk criterion, i.e., we determine a set $\mathcal{M}\subseteq\TT$ of
(up to a constant) minimal cardinality with
\begin{align*}
  \theta \est^2 \leq \sum_{T\in\mathcal{M}} \est(T)^2.
\end{align*}
For the experiments the marking parameter $\theta$ is set to $\tfrac14$.

Convergence rates in the figures are indicated by triangles, where the number $\alpha$ besides the triangle
denotes the experimental rate $\OO( (\#\TT)^{-\alpha})$. 
For uniform refinement we have $h^{2\alpha} \simeq \#\TT^{-\alpha}$.

\subsection{Smooth solution}\label{ex:smooth}
\begin{figure}
  \begin{center}
    \begin{tikzpicture}
\begin{loglogaxis}[
    title={\eqref{eq:VI:a} with $K_h^s$},
width=0.49\textwidth,
cycle list/Dark2-6,
cycle multiindex* list={
mark list*\nextlist
Dark2-6\nextlist},
every axis plot/.append style={ultra thick},
xlabel={number of elements $\#\TT$},
grid=major,
legend entries={\tiny $\err_U$,\tiny $\err_V$},
legend pos=south west,
]
\addplot table [x=nE,y=errNormU] {data/Example1Symm.dat};
\addplot table [x=nE,y=errNormV] {data/Example1Symm.dat};

\logLogSlopeTriangle{0.9}{0.2}{0.2}{0.5}{black}{{\tiny $\tfrac12$}};
\end{loglogaxis}
\end{tikzpicture}
\begin{tikzpicture}
\begin{loglogaxis}[
    title={\eqref{eq:VI:b} with $K_h^0$},
width=0.49\textwidth,
cycle list/Dark2-6,
cycle multiindex* list={
mark list*\nextlist
Dark2-6\nextlist},
every axis plot/.append style={ultra thick},
xlabel={number of elements $\#\TT$},
grid=major,
legend entries={\tiny $\err_U$,\tiny $\err_V$},
legend pos=south west,
]
\addplot table [x=nE,y=errNormU] {data/Example1NonSymm0.dat};
\addplot table [x=nE,y=errNormV] {data/Example1NonSymm0.dat};

\logLogSlopeTriangle{0.9}{0.2}{0.2}{0.5}{black}{{\tiny $\tfrac12$}};
\end{loglogaxis}
\end{tikzpicture}
\begin{tikzpicture}
\begin{loglogaxis}[
    title={\eqref{eq:VI:c} with $K_h^1$},
width=0.49\textwidth,
cycle list/Dark2-6,
cycle multiindex* list={
mark list*\nextlist
Dark2-6\nextlist},
every axis plot/.append style={ultra thick},
xlabel={number of elements $\#\TT$},
grid=major,
legend entries={\tiny $\err_U$,\tiny $\err_V$},
legend pos=south west,
]
\addplot table [x=nE,y=errNormU] {data/Example1NonSymm1.dat};
\addplot table [x=nE,y=errNormV] {data/Example1NonSymm1.dat};

\logLogSlopeTriangle{0.9}{0.2}{0.2}{0.5}{black}{{\tiny $\tfrac12$}};
\end{loglogaxis}
\end{tikzpicture}
  \end{center}
  \caption{Convergence rates for the problem from~\cref{ex:smooth}.}
  \label{fig:Example1}
\end{figure}
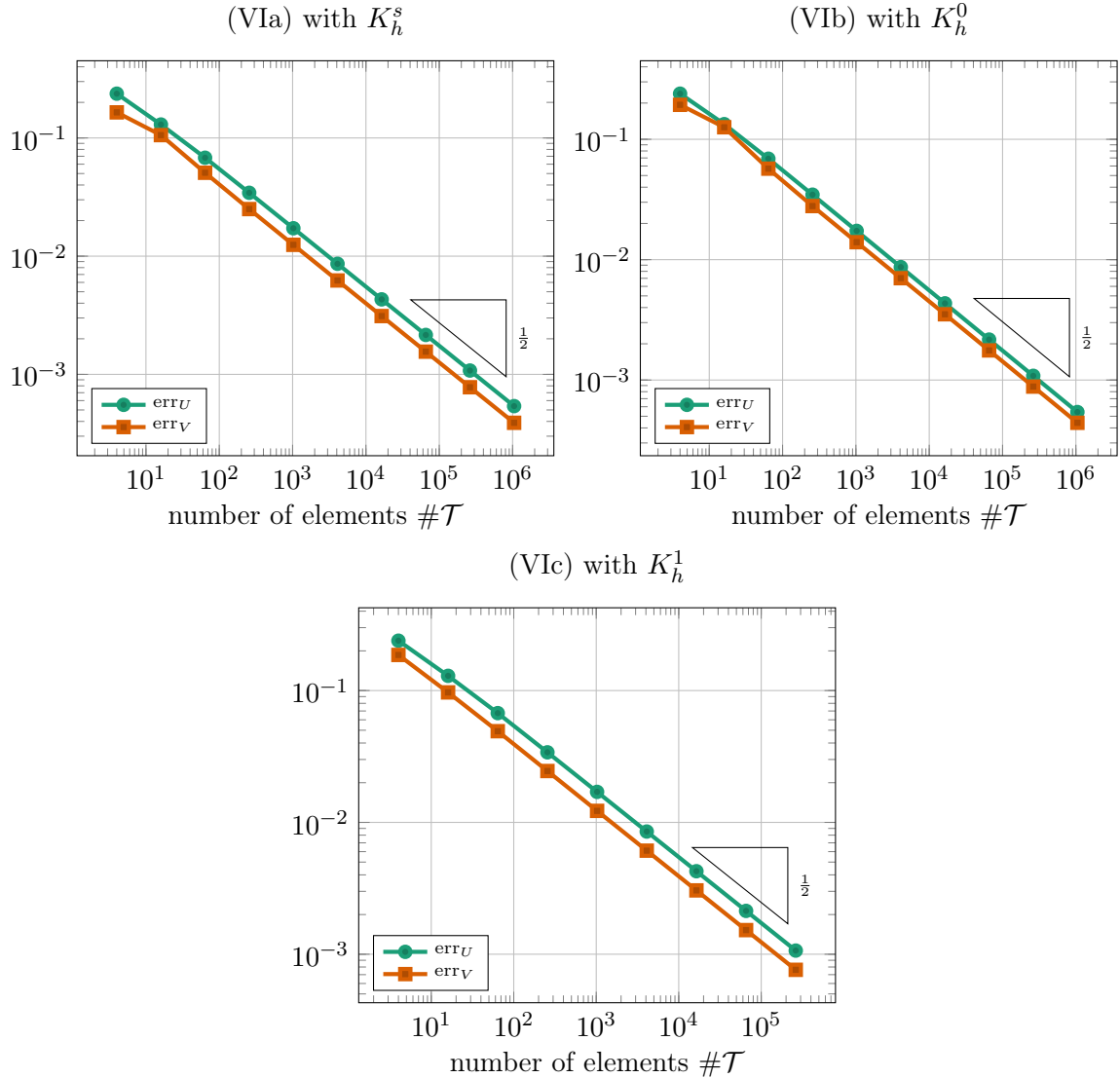

Let $\Omega = (0,1)^2$, $u(x,y) = (1-x)x(1-y)y$,
\begin{align*}
  f(x,y) := \begin{cases}
    0 & x<\tfrac12 \\
    -\Delta u (x,y) & x\geq\tfrac12
  \end{cases}.
\end{align*}
Then, $u$ solves the obstacle problem~\eqref{eq:model} with data $f$ and obstacle
\begin{align*}
  g(x,y) = \begin{cases}
    (1-x)x(1-y)y & x\leq \tfrac12 \\
    \widetilde g(x) (1-y)y & x\in (\tfrac12,\tfrac34) \\
    0 & x\geq \tfrac34
  \end{cases},
\end{align*}
where $\widetilde g$ is the unique polynomial of degree $3$ such that $g$ and $\nabla g$ are continuous at the lines
$x=\tfrac12,\tfrac34$. In particular, $g\in H^2(\Omega)$.
Note that $\lambda = -\Delta u-f \in H^1(\TT)$.
\Cref{fig:Example1} shows that the convergence rates for the solutions of the discrete variational
inequalities~\eqref{eq:VI:a}--\eqref{eq:VI:c} based on the convex sets $K_h^s$, $K_h^0$, $K_h^1$ are optimal.
This perfectly fits to our theoretic considerations in~\crefrange{thm:VI:a:conv}{thm:VI:c:conv}.
Additionally, we plot $\err_V$ which is in all cases slightly smaller than $\err_U$ but of the same order.
Note that since $\lambda$ is a $\TT$-elementwise polynomial, an inverse inequality shows that
$h\norm{\lambda-\lambda_h}{} \lesssim \norm{\lambda-\lambda_h}{-1}$ and thus $\err_V$ is equivalent to
$\norm{\uu-\uu_h}V$.

\subsection{Manufactured solution on L-shaped domain}\label{ex:Lshape}

\begin{figure}
  \begin{center}
    \begin{tikzpicture}
\begin{loglogaxis}[
    title={errors and estimator},
width=0.65\textwidth,
cycle list/Dark2-6,
cycle multiindex* list={
mark list*\nextlist
Dark2-6\nextlist},
every axis plot/.append style={ultra thick},
xlabel={number of elements $\#\TT$},
grid=major,
legend entries={\small $\err_U$ adap.,\small $\err_V$ adap., \small $\est$ adap., \small $\err_U$ unif.,\small $\err_V$
  unif., \small $\est$
unif.},
legend pos=south west,
]
\addplot table [x=nE,y=errNormU] {data/Example2Symm.dat};
\addplot table [x=nE,y=errNormV] {data/Example2Symm.dat};
\addplot table [only marks,x=nE,y=est] {data/Example2Symm.dat};
\addplot table [x=nE,y=errNormU] {data/Example2SymmUnif.dat};
\addplot table [x=nE,y=errNormV] {data/Example2SymmUnif.dat};
\addplot table [only marks,x=nE,y=est] {data/Example2SymmUnif.dat};

\logLogSlopeTriangle{0.9}{0.2}{0.5}{0.45}{black}{{\tiny $0.45$}};
\end{loglogaxis}
\end{tikzpicture}
\begin{tikzpicture}
\begin{loglogaxis}[
    title={estimator and error contributions},
width=0.65\textwidth,
cycle list/Dark2-6,
cycle multiindex* list={
mark list*\nextlist
Dark2-6\nextlist},
every axis plot/.append style={ultra thick},
xlabel={number of elements $\#\TT$},
grid=major,
legend entries={\small $\eta$,\small $\rho$, \small $\oscf$, \small $\norm{\nabla(u-u_h)}{}$,\small
$\norm{\ssigma-\ssigma_h}{}$, \small $\norm{\divergence\ssigma_h+\lambda_h+f}{}$},
legend pos=south west,
]
\addplot table [x=nE,y=eta] {data/Example2Symm.dat};
\addplot table [x=nE,y=estContact] {data/Example2Symm.dat};
\addplot table [x=nE,y=oscF] {data/Example2Symm.dat};
\addplot table [x=nE,y=errU] {data/Example2Symm.dat};
\addplot table [x=nE,y=errSigma] {data/Example2Symm.dat};
\addplot table [only marks,x=nE,y=errDivSigmaLambda] {data/Example2Symm.dat};

\logLogSlopeTriangle{0.9}{0.2}{0.7}{0.5}{black}{{\tiny $\tfrac12$}};
\end{loglogaxis}
\end{tikzpicture}
  \end{center}
  \caption{Convergence rates for the problem from~\cref{ex:Lshape}. The upper plot shows the total errors and estimators
for uniform and adaptive refinement. The lower plot compares the error and estimator contributions in the case of
adaptive refinements.}
  \label{fig:Example2Symm}
\end{figure}

We consider the same problem as given in~\cite[Section~5.2]{CarstensenBartels04}, where $g=0$, $\Omega =
(-2,2)^2\setminus [0,2]^2$ and 
\begin{align*}
  f(r,\varphi) := -r^{2/3} \sin(2/3\varphi)(\gamma'(r)/r + \gamma''(r)) - 4/3 r^{-1/3} \gamma'(r)
  \sin(2/3\varphi)-\delta(r),
\end{align*}
where $(r,\varphi)$ denote polar coordinates and $\gamma,\delta$ are given by
\begin{align*}
  \gamma(r) := \begin{cases}
    1 & r_* < 0, \\
    -6r_*^5 + 15r_*^4 -10r_*^3+1 & 0\leq r_* < 1, \\
    0 & 1\leq r_*,
  \end{cases}
\end{align*}
$r_* = 2(r-1/4)$, and
\begin{align*}
  \delta(r) := \begin{cases}
    0 & r \leq 5/4, \\
    1 & r > 5/4.
  \end{cases}
\end{align*}
The exact solution then reads $u(r,\varphi) = r^{2/3}\sin(2/3\varphi)\gamma(r)$.
Note that $u$ has a generic singularity at the reentrant corner.
We consider the discrete version of~\eqref{eq:VI:a}, where solutions are sought in the 
convex set $K_h^s$.
We conducted various tests with $\beta$ between $1$ and $100$ and the results were in all cases comparable.
For the results displayed here we have used $\beta=3$.
\Cref{fig:Example2Symm} displays convergence rates in the case of uniform and adaptive mesh-refinement.
We note that in the first plot the lines for $\err_U$ and $\est$ are almost identical.
In the second plot we compare the contributions of the overall error and estimator in the adaptive case. 
The lines for $\oscf$ and $\norm{\divergence\ssigma_h+\lambda_h+f}{}$ are almost identical.
This means that the estimator contribution $\norm{\divergence\ssigma_h+\lambda_h+\Pi_hf}{}$ in $\eta$ is negligible
and $\oscf$ is dominating the overall estimator.
We observe from the first plot that $\err_V$ is much smaller than $\err_U$ but has the same rate of convergence.
In the uniform case we see that the errors and estimators approximately converge at rate $0.45$.
One would expect a smaller rate due to the singularity. However, in this example the solution has a large gradient so
that the algorithm first refines the regions where the gradient resp. $f$ is large. 
This preasymptotic behavior was also observed in~\cite[Section~5.2]{CarstensenBartels04}.
Nevertheless, adaptivity yields a significant error reduction.

\subsection{Unknown solution}\label{ex:pyramid}

\begin{figure}
  \begin{center}
    \begin{tikzpicture}
\begin{loglogaxis}[
    title={estimators},
width=0.49\textwidth,
cycle list name=exotic,
every axis plot/.append style={ultra thick},
xlabel={number of elements $\#\TT$},
grid=major,
legend entries={\tiny $\est$ adap., \tiny $\est$ unif.},
legend pos=south west,
]
\addplot table [x=nE,y=est] {data/Example3.dat};
\addplot table [x=nE,y=est] {data/Example3Unif.dat};

\logLogSlopeTriangle{0.9}{0.2}{0.28}{0.33}{black}{{\tiny $\tfrac13$}};
\logLogSlopeTriangleBelow{0.75}{0.2}{0.08}{0.5}{black}{{\tiny $\tfrac12$}};
\end{loglogaxis}
\end{tikzpicture}
\begin{tikzpicture}
\begin{loglogaxis}[
    title={contributions},
width=0.49\textwidth,
cycle list/Dark2-6,
cycle multiindex* list={
mark list*\nextlist
Dark2-6\nextlist},
every axis plot/.append style={ultra thick},
xlabel={number of elements $\#\TT$},
grid=major,
legend entries={\tiny $\eta$ adap.,\tiny $\rho$ adap.,\tiny $\eta$ unif.,\tiny $\rho$ unif.},
legend pos=south west,
]
\addplot table [x=nE,y=eta] {data/Example3.dat};
\addplot table [x=nE,y=estContact] {data/Example3.dat};
\addplot table [x=nE,y=eta] {data/Example3Unif.dat};
\addplot table [x=nE,y=estContact] {data/Example3Unif.dat};

\logLogSlopeTriangle{0.9}{0.2}{0.45}{0.33}{black}{{\tiny $\tfrac13$}};
\logLogSlopeTriangleBelow{0.75}{0.2}{0.1}{0.5}{black}{{\tiny $\tfrac12$}};
\end{loglogaxis}
\end{tikzpicture}
  \end{center}
  \caption{Experimental convergence rates for the problem from~\cref{ex:pyramid}.}
  \label{fig:Example3}
\end{figure}
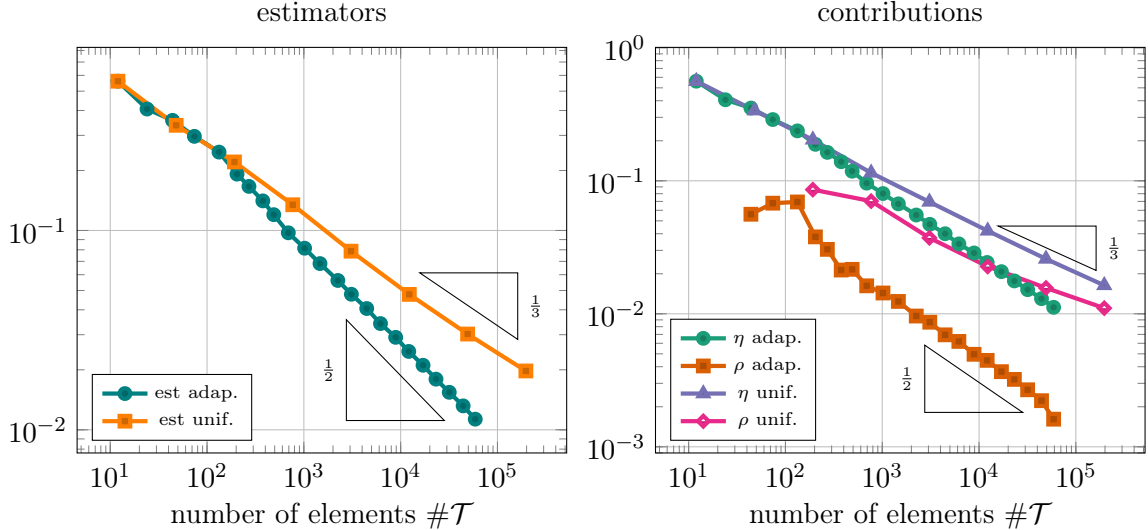

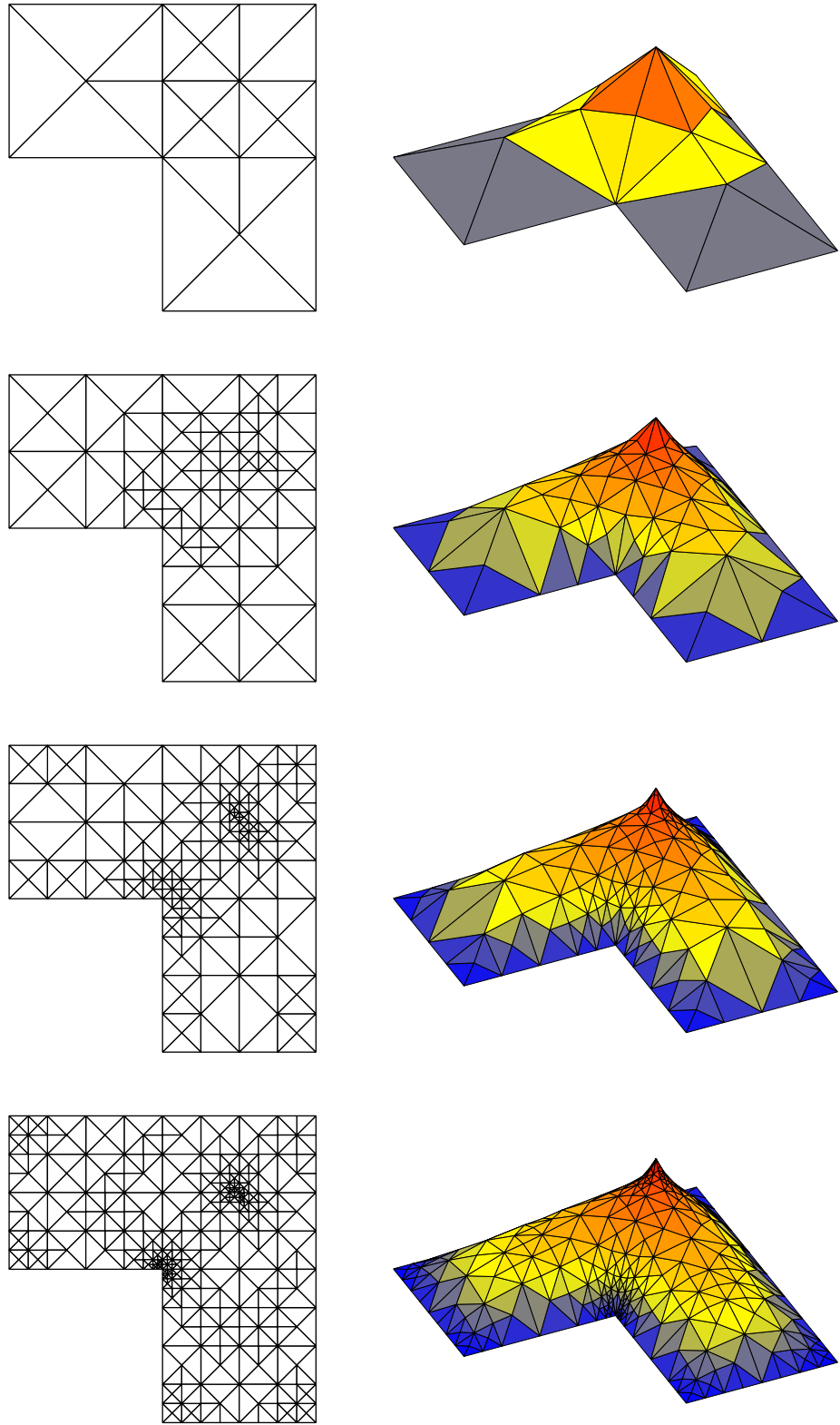
\begin{figure}
  \begin{center}
    \begin{tikzpicture}
\begin{axis}[hide axis,
width=0.49\textwidth,
    axis equal,
]

\addplot[patch,color=white,
faceted color = black, line width = 0.5pt,
patch table ={data/ele1.dat}] file{data/coo1.dat};
\end{axis}
\end{tikzpicture}
\begin{tikzpicture}
  \begin{axis}[hide axis,
width=0.49\textwidth,
view={-25}{60},
]
\addplot3[patch,line width=0.2pt,faceted color=black] table{data/sol1.dat};
\end{axis}
\end{tikzpicture}
\begin{tikzpicture}
\begin{axis}[hide axis,
width=0.49\textwidth,
    axis equal,
]

\addplot[patch,color=white,
faceted color = black, line width = 0.5pt,
patch table ={data/ele2.dat}] file{data/coo2.dat};
\end{axis}
\end{tikzpicture}
\begin{tikzpicture}
  \begin{axis}[hide axis,
width=0.49\textwidth,
view={-25}{60},
]
\addplot3[patch,line width=0.2pt,faceted color=black] table{data/sol2.dat};
\end{axis}
\end{tikzpicture}
\begin{tikzpicture}
\begin{axis}[hide axis,
width=0.49\textwidth,
    axis equal,
]

\addplot[patch,color=white,
faceted color = black, line width = 0.5pt,
patch table ={data/ele3.dat}] file{data/coo3.dat};
\end{axis}
\end{tikzpicture}
\begin{tikzpicture}
  \begin{axis}[hide axis,
width=0.49\textwidth,
view={-25}{60},
]
\addplot3[patch,line width=0.2pt,faceted color=black] table{data/sol3.dat};
\end{axis}
\end{tikzpicture}
\begin{tikzpicture}
\begin{axis}[hide axis,
width=0.49\textwidth,
    axis equal,
]

\addplot[patch,color=white,
faceted color = black, line width = 0.5pt,
patch table ={data/ele4.dat}] file{data/coo4.dat};
\end{axis}
\end{tikzpicture}
\begin{tikzpicture}
  \begin{axis}[hide axis,
width=0.49\textwidth,
view={-25}{60},
]
\addplot3[patch,line width=0.2pt,faceted color=black] table{data/sol4.dat};
\end{axis}
\end{tikzpicture}
  \end{center}
  \caption{Adaptively refined meshes and corresponding solution component $u_h$ for the problem from~\cref{ex:pyramid}.}
  \label{fig:Example3:meshSol}
\end{figure}

For our final experiment, we choose $\Omega = (-1,1)^2 \setminus [-1,0]^2$, $f=1$, and the pyramid-like obstacle
$g(x) = \max\{0,\dist(x,\partial\Omega_u)-\tfrac14\}$, where $\Omega_u = (0,1)^2$.
The solution in this case is unknown. We solve the discrete version of~\eqref{eq:VI:a} with convex set $K_h^s$.
Since $f$ is constant we have $\oscf = 0$.
\Cref{fig:Example3} shows the overall estimator (left) and its contributions (right). We observe that uniform refinement
leads to the reduced rate $\tfrac13$, whereas for adaptive refinement we recover the optimal rate.
Heuristically, we expect the solution to have a singularity at the reentrant corner as well as in the contact regions.
This would explain the reduced rates. 
\Cref{fig:Example3:meshSol} visualizes meshes produced by the adaptive algorithm and corresponding solution components
$u_h$. We observe strong refinements towards the corner $(0,0)$ and around the point $(\tfrac12,\tfrac12)$, which
coincides with the tip of the pyramid obstacle.

\section{Conclusions}
\label{sec:conclusions}

We derived a least-squares method for the classical obstacle problem and provided an a priori and a
posteriori analysis.
Moreover, we introduced and studied different variational inequalities using related bilinear forms.
All our methods are based on the first-order reformulation of the obstacle problem and provide approximations of the
displacement, its gradient and the reaction force.

\bibliographystyle{abbrv}
\bibliography{references}

\begin{thebibliography}{10}

\bibitem{StarkeLSQSignoriniS}
F.~S. Attia, Z.~Cai, and G.~Starke.
\newblock First-order system least squares for the {S}ignorini contact problem
  in linear elasticity.
\newblock {\em SIAM J. Numer. Anal.}, 47(4):3027--3043, 2009.

\bibitem{BanzSchroeder}
L.~Banz and A.~Schr\"oder.
\newblock Biorthogonal basis functions in {$hp$}-adaptive {FEM} for elliptic
  obstacle problems.
\newblock {\em Comput. Math. Appl.}, 70(8):1721--1742, 2015.

\bibitem{BanzStephan}
L.~Banz and E.~P. Stephan.
\newblock A posteriori error estimates of {$hp$}-adaptive {IPDG}-{FEM} for
  elliptic obstacle problems.
\newblock {\em Appl. Numer. Math.}, 76:76--92, 2014.

\bibitem{CarstensenBartels04}
S.~Bartels and C.~Carstensen.
\newblock Averaging techniques yield reliable a posteriori finite element error
  control for obstacle problems.
\newblock {\em Numer. Math.}, 99(2):225--249, 2004.

\bibitem{BochevGunzbergerOverview}
P.~Bochev and M.~Gunzburger.
\newblock Least-squares finite element methods.
\newblock In {\em International {C}ongress of {M}athematicians. {V}ol. {III}},
  pages 1137--1162. Eur. Math. Soc., Z\"urich, 2006.

\bibitem{BochevGunzburgerLSQ}
P.~B. Bochev and M.~D. Gunzburger.
\newblock {\em Least-squares finite element methods}, volume 166 of {\em
  Applied Mathematical Sciences}.
\newblock Springer, New York, 2009.

\bibitem{Braess05}
D.~Braess.
\newblock A posteriori error estimators for obstacle problems---another look.
\newblock {\em Numer. Math.}, 101(3):415--421, 2005.

\bibitem{BPLdiscrete}
J.~H. Bramble, R.~D. Lazarov, and J.~E. Pasciak.
\newblock A least-squares approach based on a discrete minus one inner product
  for first order systems.
\newblock {\em Math. Comp.}, 66(219):935--955, 1997.

\bibitem{BurmanHLSLSQobstacle}
E.~Burman, P.~Hansbo, M.~G. Larson, and R.~Stenberg.
\newblock Galerkin least squares finite element method for the obstacle
  problem.
\newblock {\em Comput. Methods Appl. Mech. Engrg.}, 313:362--374, 2017.

\bibitem{ChenNochetto00}
Z.~Chen and R.~H. Nochetto.
\newblock Residual type a posteriori error estimates for elliptic obstacle
  problems.
\newblock {\em Numer. Math.}, 84(4):527--548, 2000.

\bibitem{ChoulyHildNitsche}
F.~Chouly and P.~Hild.
\newblock A {N}itsche-based method for unilateral contact problems: numerical
  analysis.
\newblock {\em SIAM J. Numer. Anal.}, 51(2):1295--1307, 2013.

\bibitem{Falk74}
R.~S. Falk.
\newblock Error estimates for the approximation of a class of variational
  inequalities.
\newblock {\em Math. Comput.}, 28:963--971, 1974.

\bibitem{DPGsignorini}
T.~F\"uhrer, N.~Heuer, and E.~P. Stephan.
\newblock On the {DPG} method for {S}ignorini problems.
\newblock {\em IMA Journal of Numerical Analysis}, page in print, 2017.

\bibitem{Glowinski08}
R.~Glowinski.
\newblock {\em Numerical methods for nonlinear variational problems}.
\newblock Scientific Computation. Springer-Verlag, Berlin, 2008.
\newblock Reprint of the 1984 original.

\bibitem{Glowinski81}
R.~Glowinski, J.-L. Lions, and R.~Tr{\'e}moli{\`e}res.
\newblock {\em Numerical analysis of variational inequalities}, volume~8 of
  {\em Studies in Mathematics and its Applications}.
\newblock North-Holland Publishing Co., Amsterdam-New York, 1981.
\newblock Translated from the French.

\bibitem{GustafssonStenbergVideman_MixedFEMObstacle}
T.~Gustafsson, R.~Stenberg, and J.~Videman.
\newblock Mixed and stabilized finite element methods for the obstacle problem.
\newblock {\em SIAM J. Numer. Anal.}, 55(6):2718--2744, 2017.

\bibitem{MR3667082}
T.~Gustafsson, R.~Stenberg, and J.~Videman.
\newblock On finite element formulations for the obstacle problem---mixed and
  stabilised methods.
\newblock {\em Comput. Methods Appl. Math.}, 17(3):413--429, 2017.

\bibitem{MR1276702}
R.~H.~W. Hoppe and R.~Kornhuber.
\newblock Adaptive multilevel methods for obstacle problems.
\newblock {\em SIAM J. Numer. Anal.}, 31(2):301--323, 1994.

\bibitem{KarkkainenKT_03_ALA}
T.~K{\"a}rkk{\"a}inen, K.~Kunisch, and P.~Tarvainen.
\newblock Augmented {L}agrangian active set methods for obstacle problems.
\newblock {\em J. Optim. Theory Appl.}, 119(3):499--533, 2003.

\bibitem{Karkulik2013}
M.~Karkulik, D.~Pavlicek, and D.~Praetorius.
\newblock On {2D} {N}ewest {V}ertex {B}isection: {O}ptimality of
  {M}esh-{C}losure and {$H^1$}-{S}tability of {$L_2$}-{P}rojection.
\newblock {\em Constr. Approx.}, 38(2):213--234, 2013.

\bibitem{KinderlehrerStampacchia}
D.~Kinderlehrer and G.~Stampacchia.
\newblock {\em An introduction to variational inequalities and their
  applications}, volume~31 of {\em Classics in Applied Mathematics}.
\newblock Society for Industrial and Applied Mathematics (SIAM), Philadelphia,
  PA, 2000.
\newblock Reprint of the 1980 original.

\bibitem{StarkeLSQSignorini}
R.~Krause, B.~M\"uller, and G.~Starke.
\newblock An adaptive least-squares mixed finite element method for the
  {S}ignorini problem.
\newblock {\em Numer. Methods Partial Differential Equations}, 33(1):276--289,
  2017.

\bibitem{NochettoSV03}
R.~H. Nochetto, K.~G. Siebert, and A.~Veeser.
\newblock Pointwise a posteriori error control for elliptic obstacle problems.
\newblock {\em Numer. Math.}, 95(1):163--195, 2003.

\bibitem{NochettoSiebertVeeser05}
R.~H. Nochetto, K.~G. Siebert, and A.~Veeser.
\newblock Fully localized a posteriori error estimators and barrier sets for
  contact problems.
\newblock {\em SIAM J. Numer. Anal.}, 42(5):2118--2135, 2005.

\bibitem{Rodrigues_87_OPM}
J.-F. Rodrigues.
\newblock {\em Obstacle problems in mathematical physics}, volume 134 of {\em
  North-Holland Mathematics Studies}.
\newblock North-Holland Publishing Co., Amsterdam, 1987.
\newblock Notas de Matem{\'a}tica [Mathematical Notes], 114.

\bibitem{stevenson:NVB}
R.~Stevenson.
\newblock The completion of locally refined simplicial partitions created by
  bisection.
\newblock {\em Math. Comp.}, 77(261):227--241, 2008.

\bibitem{Veeser01I}
A.~Veeser.
\newblock Efficient and reliable a posteriori error estimators for elliptic
  obstacle problems.
\newblock {\em SIAM J. Numer. Anal.}, 39(1):146--167, 2001.

\bibitem{WeissWohlmuth10}
A.~Weiss and B.~I. Wohlmuth.
\newblock A posteriori error estimator for obstacle problems.
\newblock {\em SIAM J. Sci. Comput.}, 32(5):2627--2658, 2010.

\end{thebibliography}
\end{document}